\documentclass[11pt, twoside]{article}
\usepackage{amsmath,amsthm,amssymb}
\usepackage{enumerate}
\usepackage{graphicx}
\usepackage{color}
\usepackage[colorlinks=true]{hyperref}

\def\titlerunning#1{\gdef\titrun{#1}}

\makeatletter

\def\author#1{\gdef\autrun{\def\and{\unskip, }#1}\gdef\@author{#1}}

\makeatother

\def\email#1{e-mail: #1}

\def\subjclass#1{{\renewcommand{\thefootnote}{}%
\footnote{\emph{Mathematics Subject Classification (2010):} #1}}}

\def\keywords#1{\par\medskip
\noindent\textbf{Keywords.} #1}

\newtheorem{theorem}{Theorem}[section]
\newtheorem{lemma}[theorem]{Lemma}
\newtheorem{proposition}[theorem]{Proposition}

\hypersetup{linkcolor=blue,urlcolor=red,citecolor=red}

\theoremstyle{definition}

\newtheorem{remark}[theorem]{Remark}

\numberwithin{equation}{section}

\begin{document}

\baselineskip=17pt

\titlerunning{}

\title{Flow decomposition for heat equations with memory}

\author{Gengsheng Wang\thanks{Center for Applied Mathematics, Tianjin University,
		Tianjin, 300072, China. \email{wanggs62@yeah.net}}
\and
Yubiao Zhang\thanks{Center for Applied Mathematics, Tianjin University, Tianjin, 300072, China; \email{yubiao\b{ }zhang@yeah.net}}
\and
Enrique Zuazua\thanks{
	[1] Chair for Dynamics, Control and Numerics - Alexander von Humboldt-Professorship, Department of Data Science,  Friedrich-Alexander-Universit\"at Erlangen-N\"urnberg,
	91058 Erlangen, Germany (\email{\texttt{enrique.zuazua@fau.de}}),
	\newline \indent\hskip 7pt
	[2] Chair of Computational Mathematics, Fundaci\'{o}n Deusto,
	Av. de las Universidades, 24,
	48007 Bilbao, Basque Country, Spain,
	\newline \indent\hskip 7pt
	[3] Departamento de Matem\'{a}ticas,
	Universidad Aut\'{o}noma de Madrid,
	28049 Madrid, Spain
}
}

\date{}

\maketitle


\begin{abstract}
We build up a decomposition
for the  flow generated by
  the heat equation with a real analytic  memory kernel.
It consists of three components: The first  one is of parabolic nature; the second one gathers the hyperbolic component of the dynamics, with null velocity of propagation; the last one  exhibits  a finite  smoothing effect.
  This decomposition
     reveals the  hybrid parabolic-hyperbolic nature of the flow and clearly illustrates the significant impact of the memory term on the parabolic behavior of the system in the absence of memory terms.

\end{abstract}


\section{Introduction}

\subsection{Statement of the problem}

In this paper, we will  study  the following  heat equation with  memory:
\begin{eqnarray}\label{our-system}
\left\{
\begin{array}{ll}
\partial_t y(t,x)  - \Delta y(t,x)
+ \displaystyle\int_0^t M(t-s) y(s,x)ds =0,
~&(t,x)\in \mathbb R^+ \times \Omega,\\
y(t,x)=0,~&(t,x)\in \mathbb R^+\times \partial\Omega,  \\
y(0,x)=y_0(x),~&\quad~~ x \in \Omega.
\end{array}
\right.
\end{eqnarray}
Here,  $\mathbb R^+:=(0,+\infty)$, $\Omega\subset\mathbb R^n$ ($n\in \mathbb N^+:=\{1,2,3,\cdots\}$)  is a bounded domain with a $C^2$-boundary $\partial\Omega$,
$y_0$ is  an initial datum and $M$ is a memory-kernel  over $\overline{\mathbb R^+}:=[0,+\infty)$.

Although our analysis can be generalized to less regular memory kernels, for the sake of simplicity we assume that:
\begin{description}
  \item[ ] ($\mathfrak C$)~~~~
 the memory kernel
  $M$ is  a real analytic and nonzero function over $\overline{\mathbb R^+}$.
\end{description}
Equations with memory arise in the modeling of many  physical phenomena such as
viscoelasticity,  heat conduction, etc.
They can be  traced back to the works of J. Maxwell  \cite{Maxwell}, L. Boltzmann \cite{Boltzmann-1,Boltzmann-2} and V. Volterra \cite{Volterra-1,Volterra-2}.  For instance, in the analysis of elastic materials, L. Boltzmann and V. Volterra  represented the stress tensor in terms of the strain tensor as well as its history values.
Equations involving memory terms have been  widely studied: see for instance
\cite{Amendola-Fabriio-Golden,Cattaneo,  Chaves-Silva-Zhang-Zuazua,   Christensen,Coleman-Gurtin,
Dafermos,Farbizio-Gogi-Pata,Fu-Yong-Zhang,Gurtin-Pipkin, Lu-Zhang-Zuazua, Pandolfi} and the references therein.
 In particular, in \cite{Gurtin-Pipkin} the general memory effect in heat conduction processes was analyzed showing that temperature waves travelling in the direction of the heat-flux  propagate faster than wave travelling in the opposite direction, while in \cite{Dafermos}
 the asymptotic behavior of the systems of linear viscoelasticity at large time was analyzed,
    introducing a new auxiliary variable to deal with  the history of the states.

By standard methods (see, for instance, \cite[Theorem 1.2 in Section 6.1, p. 184]{Pazy}),  it can be shown that the
 equation \eqref{our-system}, with $y_0\in L^2(\Omega)$, has a unique mild solution, denoted by $y(\cdot;y_0)$, in the space $C(\overline{\mathbb R^+}; L^2(\Omega) )$.
For each $t\geq 0$, we let the evolution of the system be denoted by:
\begin{align}\label{varPhi-y-y0}
\varPhi(t)y_0:=y(t;y_0),\;\; y_0 \in L^2(\Omega).
\end{align}
For each $t\ge 0$,   the flow  generated by the
equation \eqref{our-system}, $\varPhi(t)$, belongs to $\mathcal L( L^2(\Omega) )$. Here and in what follows,
 we denote by $\mathcal L(E,F)$ (where $E$ and $F$ are two Banach  spaces) the space of all linear
 and bounded  operators from  $E$ to  $F$, and  simply write $\mathcal L(E)$ for $\mathcal L(E,E)$.

We shall use the notation $\{e^{tA}\}_{t\geq 0}$ for the $C_0$ semigroup generated by the heat equation in the absence of memory term (i.e., when $M \equiv 0$), where
\begin{eqnarray}\label{selfadjoit-elliptic-operator}
A f := \Delta f,
\;\;\mbox{with its domain}\;\;
D(A):=H^2(\Omega)\cap H_0^1(\Omega).
\end{eqnarray}
Then $z(t;y_0):=e^{tA}y_0,\; t\geq 0$, solves \eqref{our-system} without memory, i.e., \eqref{our-system} when  $M \equiv 0$.

This paper is devoted to analyzing the dynamics of the system with memory term and, in particular, to exhibiting the significant differences with the heat semigroup in the absence of memory.

\subsection{Main results}

The aim of this paper is to build up a decomposition of the flow $\varPhi(t)$, revealing a   hybrid parabolic-hyperbolic
 dynamics of the  \eqref{our-system}.

To state our main results, we first introduce several  concepts, definitions and notations.
\begin{itemize}

\item Let $\eta_j>0$  be the $j^{\text{th}}$ eigenvalue of $-A$ and let $e_j$
be the corresponding normalized eigenfunction in $L^2(\Omega)$.
Define, for each $s\in \mathbb{R}$,
  the real Hilbert space:
\begin{eqnarray}\label{def-space-with-boundary-condition}
\mathcal H^s :=
\Big\{f= \sum_{j=1}^\infty a_j e_j ~:~ (a_j)_{j\geq 1}\subset \mathbb R,~
\sum_{j=1}^\infty |a_j|^2 \eta_j^{s}<+\infty
\Big\},
\end{eqnarray}
equipped with the inner product:
\begin{eqnarray*}
\langle f_1, f_2\rangle_{\mathcal H^s}   :=
\sum_{j=1}^\infty a_{j,1} a_{j,2} \eta_j^{s},\;
\;\;\;\;
f_k= \sum_{j=1}^\infty a_{j,k} e_j
\in
\mathcal H^s~(k=1,2).
\end{eqnarray*}

For all  $t\geq 0$, $\varPhi(t)$ belongs to $\mathcal L( \mathcal H^s)$  for any $s\in \mathbb{R}$ (see
 Proposition \ref{prop5.1}).

We now introduce the classes:
 \begin{align}\label{1.4,2-10}
\mathcal H^{-\infty}:=\displaystyle\bigcup_{s\in \mathbb R} \mathcal H^{s}
~~\text{and}~~
\mathcal H^{+\infty}:= \displaystyle\bigcap_{s\in \mathbb R} \mathcal H^{s}.
\end{align}

\item Recall  that  for each continuous function $f$ over $\mathbb R^+$,  the operator $f(-A)$ can be defined by the spectral functional calculus
  (see \cite[Section 3, Chapter V.III]{SIMON1}).

  \item Let us also define the following functions,  related to the memory kernel $M$, that will play  important roles in the decomposition of the flow.

  First, we introduce the {\it flow kernel}:
  \begin{align}\label{new-kernel-KM}
	K_M(t,s):=
	\sum_{j=1}^{+\infty}  \frac{ (-s)^j }{ j! } \underset{j}{ \underbrace{M*\cdots*M} }(t-s),~~&(t,s)\in S_+:= \big\{(t,s)\in \mathbb R^2 ~:~ t\geq s \big\}.
	\end{align}
	Here and throughout the paper, $*$ denotes the usual convolution, i.e., when $g_1,g_2\in L^1_{loc}(\overline{ \mathbb R^+ })$,
 $$g_1 * g_2 (t) :=  \int_0^t  g_1(t-s) g_2(s) ds, \, t\geq 0.$$
 Notice that the above  $K_M$ is well-defined and  it
 is real analytic over $S_+$
 (see Proposition \ref{prop-KM-analytic}). The following  holds (see Proposition \ref{prop-varPhi-expression})
	\begin{align*}
	\varPhi(t)
	= e^{tA} +   \int_0^t K_M(t,\tau) e^{\tau A} d\tau,
	\;\;  t\geq 0,
	\end{align*}
which yields a clear description of the gap between the heat equation and the memory one and justifies the terminology ``flow kernel" employed.

  Second, for each $N\in\mathbb N^+$, let
  \begin{align}\label{0921-RN-good-remainder}
	R_N(t,\tau) := \int_0^t \tau e^{-\tau s} \partial_s^N K_M(t,s)  ds,~~
	t\geq0,~~\tau\geq 0.
	\end{align}

    Third,  we define  two sequences of functions $\{h_l\}_{l\in\mathbb N}$ and  $\{p_l\}_{l\in \mathbb N}$ (that will play the role of coefficients in the expansions) in the following manner: for each $t \ge 0$,
\begin{align}\label{thm-ODE-meomery-asymptotic-estimate-hypobolic}
\begin{cases}
h_l(t):=&
(-1)^l \displaystyle\sum_{j=0}^l  C_{l}^{l-j}
\dfrac{ d^{(l-j)} }{ dt^{(l-j)} } \underset{j}{ \underbrace{M*\cdots*M} } (t);
\\
p_l(t):=& - h_l(0)+ (-1)^{l+1}
\displaystyle\sum_{  \tiny\begin{array}{c}
	m,j\in \mathbb N^+, \\
	2j-l-1\leq m \leq j
	\end{array}
}
\bigg(
C_{l}^{l-j+m}
\dfrac{ d^{(l-j+m)} }{ dt^{(l-j+m)} } \underset{j}{ \underbrace{M*\cdots*M} }(0)
\bigg)  \frac{(-t)^m}{m!}.
\end{cases}
\end{align}
Here,
$C_\beta^m:=\beta!/m!(\beta-m)!$
 denotes the binomial coefficients
 and
  $\underset{j}{ \underbrace{M*\cdots*M} }:=0$ when $j=0$.
\item Let $f$ be a distribution  over a domain $\mathcal D\subset \mathbb R^k$ (with $k\in \mathbb{N}^+$).
By
the notation  $f\in L^2_{loc}(q)$ (with $q\in \mathcal D)$, we refer to the fact that $f|_{ \mathcal U_0 } \in L^2( \mathcal U_0 )$ for an open non-empty subset $\mathcal U_0$ such that $q\in \mathcal U_0 \subset \mathcal D$.

\end{itemize}

The main results of this paper are   as follows.

\begin{theorem}\label{cor-0423-demcomposition}
	For each integer  $N\geq 2$, the flow $\varPhi(t)$ admits  the following decomposition:
\begin{align}\label{0423-demcomposition-eq}
	\varPhi(t) =  \mathcal P_N(t)+  \mathcal W_N(t)  +  \mathfrak R_N(t),~~t\geq0,
\end{align}
with
 \begin{align}\label{def-PN-HN-RN}
\left\{
	\begin{array}{lll}
	  \mathcal P_N(t)&:=& e^{tA} + e^{tA}\sum_{l=0}^{N-1} p_{l}(t)  (-A)^{-l-1}   ,\\
	  \mathcal W_N(t)&:=&\sum_{l=0}^{N-1} h_{l}(t)  (-A)^{-l-1},\\
	  \mathfrak R_N(t)&:=&R_N(t,-A) (-A)^{-N-1},
	\end{array}
~~~~~	t\geq0,
\right.
\end{align}
where  $\{h_l\}_{l\in\mathbb{N}}$ and $\{p_l\}_{l\in\mathbb{N}}$ are given by \eqref{thm-ODE-meomery-asymptotic-estimate-hypobolic}
and $R_N$ is given by \eqref{0921-RN-good-remainder}.
Moreover, for each $t\geq 0$,  neither $\{h_l(t)\}_{l\geq1}$ nor  $\{p_l(t)\}_{l\in\mathbb{N}}$
	 is the null sequence,  i.e.,
\begin{align}\label{2021-april-hl-pl-nontrivial}
 \sum_{l\geq 1} |h_l(t)| >0
 ~~\text{and}~~
 \sum_{l\geq 1}  |p_l(t)|>0.
\end{align}

\end{theorem}

\begin{theorem}\label{210320-yb-thm-main-explanations}
 With the notation in Theorem \ref{cor-0423-demcomposition},
 the following conclusions are true for each integer $N\geq 2$:
\begin{enumerate}

\item[(i)] The first component  $\mathcal P_N$ exhibits a   heat-like behavior:
for each $t>0$,  $\mathcal P_N(t) \mathcal H^{-\infty}\subset \mathcal H^{+\infty}$, where $\mathcal H^{+\infty}$
and $\mathcal H^{-\infty}$
are given by \eqref{1.4,2-10}.

\item[(ii)] The second component $\mathcal W_N$
exhibits a  wave-like behavior: when $y_0\in \mathcal H^{-\infty}$, $x_0\in\Omega$ and $t_0>0$,
\begin{align}\label{yb-202104-wn-propagation}
  \mathcal W_N(\cdot)y_0\not\in L^2_{loc}(t_0,x_0)
  ~\Longleftrightarrow~
  \forall\, t>0,~
  \mathcal W_N(\cdot)y_0\not\in L^2_{loc}(t,x_0).
\end{align}

In other words, the singularities of the solutions propagate in the time-like direction with null velocity of propagation in the space-like direction.

\item[(iii)]  The last component $\mathfrak R_N$
exhibits a time-uniform  smoothing effect with a gain of $2N+2$ space derivatives:
for each  $y_0 \in \mathcal H^s$ with $s\in \mathbb R$,
$\mathfrak R_N(\cdot)y_0 \in C\big( [0,+\infty);\mathcal H^{s+2N+2} \big)$,
while $A^{-j}e^{\cdot A}y_0,A^{-j}y_0 \in C\big( [0,+\infty);\mathcal H^{s+2j} \big)$
for any $0\leq j \leq N$. And for each  $s\in \mathbb R$, the  term  $R_N$
(in $\mathfrak R_N$)
belongs to $C(\mathbb R^+;\mathcal L( \mathcal H^s))$   and fulfills the estimate:
    \begin{eqnarray}\label{RN-property-regularity}
	\big\| R_N(t,-A) \|_{ \mathcal L( \mathcal H^s) }
	\leq e^t \bigg\{
	\exp\bigg[ N(1+t) \bigg(\sum_{j=0}^N \max_{0\leq \tau\leq t} \Big| \frac{d^j}{d\tau^j}M(\tau) \Big| \bigg) \bigg]  - 1  \bigg\}
	,~~t\geq 0.
	\end{eqnarray}
\item[(iv)]
For any $y_0\in  \mathcal H^{-\infty}$,  $x_0\in \Omega$ and $t>0$,
	 \begin{eqnarray}\label{202104-varphi-wn-y0}
	    \varPhi(\cdot)y_0\not\in L^2_{loc}(t,x_0)  \Leftrightarrow
	    \mathcal W_N(\cdot)y_0\not\in L^2_{loc}(t,x_0)  \Leftrightarrow  A^{-2}y_0\not\in L^2_{loc}(x_0).
	 \end{eqnarray}
\end{enumerate}

\end{theorem}


\begin{theorem}\label{210320-yb-thm-main-high-frequency-behaviors}
With  the notation in Theorem \ref{cor-0423-demcomposition},
the following conclusions hold for each integer $N\geq 2$:
First, at the initial time,
\begin{align}\label{yb-march-initial-behaviors-limit}
\displaystyle\lim_{j\rightarrow+\infty}
\|\varPhi(0)e_j\|_{L^2(\Omega)} =
\displaystyle\lim_{j\rightarrow+\infty} \|\mathcal P_N(0)e_j\|_{L^2(\Omega)}=1,\,
\displaystyle\lim_{j\rightarrow+\infty} \|\mathcal W_N(0)e_j\|_{L^2(\Omega)}
=0,\,
\mathfrak R_N(0)=0.
\end{align}
Second,  at each time $t>0$,
\begin{align}\label{210320-yb-positive-behaviors-limit}
\left\{
\begin{array}{l}
\displaystyle\lim_{j\rightarrow+\infty}
\|\varPhi(t)e_j\|_{ \mathcal H^4 }=
\displaystyle\lim_{j\rightarrow+\infty} \|\mathcal W_N(t)e_j\|_{ \mathcal H^4 }=|M(t)|,
\\
\displaystyle\lim_{j\rightarrow+\infty}
\|\varPhi(t)e_j\|_{\mathcal H^s }=
\displaystyle\lim_{j\rightarrow+\infty} \|\mathcal W_N(t)e_j\|_{ \mathcal H^s}=0\;\;\mbox{for any}\;\;s<4,
\\
\displaystyle\lim_{j\rightarrow+\infty} \|\mathcal P_N(t)e_j\|_{ \mathcal H^s }=0\;\;\mbox{for any}\;\;s\in \mathbb{R},\\
\displaystyle\lim_{j\rightarrow+\infty} \|\mathfrak R_N(t)e_j\|_{ \mathcal H^s }=0\;\;\mbox{for any}\;\; s < 2N+2.
\end{array}
\right.
\end{align}
%
\end{theorem}

\begin{theorem}\label{210320-yb-thm-flow-L2H4}
Given $s\in \mathbb R$, the following conclusions are true:
\begin{enumerate}
\item[(i)]
 There is $C_0>0$ (independent of $s$) so that for any
 $ \alpha \in[0,4] $ and $t>0$, $\varPhi(t)$ belongs to $\mathcal L(\mathcal H^s,\mathcal H^{s+\alpha})$ and satisfies
\begin{align}\label{march-flow-H4-estimate}
   \| \varPhi(t) \|_{\mathcal L(\mathcal H^s,\mathcal H^{s+\alpha})}
   \leq     C_0 t^{ -\frac{\alpha}{2} } \exp\Big[2(1+t) \big(1+\|M\|_{C^2([0,t])} \big) \Big].
\end{align}
\item[(ii)] If there is  $\alpha_0\geq 0$  so that
\begin{align}\label{march-flow-H4-estimate-1}
   \varPhi(t) \in \mathcal L(\mathcal H^s,\mathcal H^{s+\alpha_0})
   ~~\text{as}~~
    t>0\;\;\mbox{in a neighborhood of}\;\; 0,
\end{align}
then  $\alpha_0\leq 4$
and
  $\displaystyle\liminf_{t\rightarrow 0^+}t^{\frac{\alpha_0}{2}}\| \varPhi(t) \|_{\mathcal L(\mathcal H^s,\mathcal H^{s+\alpha_0})}>0.$
\item[(iii)] For any $y_0\in \mathcal H^s$, $\varPhi(\cdot)y_0\in C(\mathbb{R}^+;\mathcal H^{s+4})$. Moreover,
the index $4$ is optimal in the sense that if $\alpha>4$, then  $\varPhi(\cdot) \hat y_0\notin C(\mathbb{R}^+;\mathcal H^{s+\alpha})$ for some $\hat y_0\in \mathcal H^s$.

\end{enumerate}
\end{theorem}

\begin{remark}\label{remark1.2-2-13}
 Several comments are in order:
\begin{itemize}
\item[(a1)]  Theorem \ref{cor-0423-demcomposition}  gives the decomposition \eqref{0423-demcomposition-eq} of the flow and
besides shows the non-triviality of $\mathcal P_N$ and $\mathcal W_N$: for each $t\geq 0$, there is $N_0(t)\in\mathbb{N}^+$
so that $\mathcal P_N(t)\neq 0$ and $\mathcal W_N(t)\neq 0$ when $N\geq N_0(t)$. Theorem \ref{210320-yb-thm-main-explanations}
explains the functionality of each term in the decomposition \eqref{0423-demcomposition-eq}.
Theorems \ref{210320-yb-thm-main-high-frequency-behaviors}-\ref{210320-yb-thm-flow-L2H4}
are the consequences of the decomposition \eqref{0423-demcomposition-eq}.
The three terms of the decomposition $\mathcal P_N$, $ \mathcal W_N$ and $\mathfrak R_N$ are referred to as the heat-like component, the wave-like component and the remainder,
respectively.
 The first two components are the leading ones. Due to their asymptotic expression, we can clearly identify their nature and this justifies the terminology heat/wave-like respectively.
(This coincides with the expected hybrid nature of the flow.)

\item[(a2)] The proof of Theorem \ref{cor-0423-demcomposition}
 uses a Fourier expansion on the basis of eigenfunctions of $A$ that reduces the problem to consider an ODE with memory depending on the dual parameter  $\eta >0$:
\begin{eqnarray}\label{ode-memory2-15}
w^\prime(t)+\eta w(t) + \int_0^t M(t-s)w(s) ds=0,~~t>0;~~~
w(0)=1.
\end{eqnarray}
The dynamics of this memory-ODE can be decomposed into three terms leading to the three components in the decomposition \eqref{0423-demcomposition-eq}.  The asymptotics of this decomposition for large  $\eta$  yields the main properties of the decomposition \eqref{0423-demcomposition-eq} of the memory-heat equation. A careful analysis of the  flow kernel $K_M$ plays a key role in this analysis.

  \item[(a3)]  In Theorem \ref{210320-yb-thm-main-explanations}, the infinite order regularizing effect of the heat-like behavior of $\mathcal P_N$,   stated  in $(i)$,   is the analog of  the smoothing effect of the semigroup $\{e^{tA}\}_{t\geq 0}$ generated by the heat equation without memory terms;
  The wave-like component  $\mathcal W_N$ exhibits
   the propagation of singularities
along the time-direction, as stated in  $(ii)$;
  The smoothing effect of the remainder $\mathfrak R_N$,   stated in $(iii)$, ensures the gain of $2N+2$ space-derivatives at nonnegative time  but differs from the  infinite order smoothing effect of the heat semigroup $e^{tA}$  at positive time (see Remark \ref{remark-heat-smoothing-effect} for more discussions).
  The conclusion $(iv)$ says, in plain language, that  when $t>0$,
the  singularity of
$\mathcal W_N(t)$ determines the singularity of $\varPhi(t)$, more precisely, the singularity of
the practical leading term in $\mathcal W_N(t)$ determines the singularity of $\varPhi(t)$. Here, we notice  that
 it follows by \eqref{thm-ODE-meomery-asymptotic-estimate-hypobolic} that $h_0(t)\equiv 0$ and $h_1(t)=-M(t)$,
 thus the practical leading term in $\mathcal W_N(t)$ is $-M(t)A^2$ where the coefficient $M(t)$ is not zero
 except for finitely many $t$ by the assumption ($\mathfrak C$).

  From these, we conclude that the decomposition \eqref{0423-demcomposition-eq} in Theorem \ref{cor-0423-demcomposition}
reveals  the hybrid parabolic-hyperbolic behavior of the flow
$\varPhi(t)$.

   \item[(a4)]  Theorem \ref{210320-yb-thm-main-high-frequency-behaviors}
shows
 how the energy of  solutions taking  eigenfunctions of the operator $A$ as  initial data  is distributed
over each component of \eqref{0423-demcomposition-eq}
 at time $t=0$ and  time $t>0$ respectively.
 The conclusion \eqref{yb-march-initial-behaviors-limit} says that
 when $t=0$,
 the energy of $\varPhi(0)e_j$ (which is exactly $e_j$) is almost concentrated in the heat-like component for large $j$,
 while \eqref{210320-yb-positive-behaviors-limit} can be explained as: when $t>0$,
 the energy of $\varPhi(t)e_j$ almost focuses  on the wave-like component for large $j$.
 The first line in \eqref{210320-yb-positive-behaviors-limit}
     is from the term $-M(t)A^{-2}$ (in $\mathcal W_N$) and the order $4$ in $\mathcal{H}^4$ is exactly from the order $-2$ in
     $A^{-2}$. The last line in \eqref{210320-yb-positive-behaviors-limit}
     is from $(iii)$ in Theorem \ref{210320-yb-thm-main-explanations}.

     \item[(a5)]   Theorem \ref{210320-yb-thm-flow-L2H4}  exhibits the finite order smoothing effect of the flow.
      By it, we can see,  from
      the point of view of the smoothing effect,
      both big differences and some
     similarities between the flow $\varPhi(t)$ and the semigroup $e^{tA}$:
     \begin{itemize}
     	\item[] First, on one hand, for any  $y_0\in\mathcal H^s$ (with $s\in \mathbb{R}$),
     	$e^{\cdot A}y_0\in C(\mathbb{R}^+;\mathcal H^k)$ ($\forall\, k\in \mathbb N$),
     	while
     	$\varPhi(\cdot)y_0\in C(\mathbb{R}^+;\mathcal H^{s+4})$ and moreover the index $4$ is optimal,
     	on the other hand,
     	the smoothing effect of the flow $\varPhi(t)$ at points  in the set $\{t>0~:~M(t)=0\}$
     	is better than that at points in the set
     	$\{t>0~:~M(t)\neq 0\}$ (see Remark \ref{remark-yb-202104-special-points} for more details);
     	
     	\item[]  Second,  when $t$ is large,  both  the semigroup $e^{tA}$ and the flow $\varPhi(t)$ are
     	bounded from the above by
     	exponential functions of $t$,
     	while when $t$ is small, they are bounded from the above by
     	the function $t^{-2}$. Moreover,
     	the flow $\varPhi(t)$ is also bounded from the below by
     	$t^{-2}$ when $t\rightarrow 0^+$.
     	Here, by  ``$e^{tA}$/$\varPhi(t)$ is bounded by", we mean
     	``the $\mathcal L(\mathcal H^s,\mathcal H^{s+4})$-norm of $e^{tA}$/$\varPhi(t)$ is bounded by".
     \end{itemize}



\item[(a6)]  It follows from  Theorem \ref{cor-0423-demcomposition} that for each $t\geq 0$,   the wave-like component
$\mathcal W_N(t)$,  as well as the gap between the heat-like component $\mathcal P_N(t)$ and the heat semigroup $e^{tA}$,
are non-trivial, when $N$ is large enough.
  Thus, we may expect that as $N$ increases
  the heat-like  and the wave-like components include an increasing number of terms,
  just like in the Taylor expansion, and the decomposition becomes sharper. (See the example in Section \ref{example}.)

\item[(a7)]  The last component $\mathfrak R_N$  does not fit completely into any of the two previous ones, but  is needed in order to complete the representation of solutions.
We further explain $\mathfrak R_N$ as follows:
\begin{itemize}
\item[$\bullet$] Hybridity. It   inherits a hybrid heat/wave structure from the flow
    (see the note (R2) in Remark \ref{remark5.2}).

	\item[$\bullet$] Smoothing effect. It has the time-uniform smoothing effect  given in ($iii$) of Theorem \ref{210320-yb-thm-main-explanations}. Such time-uniform smoothing effect differs from the usual smoothing effect of the heat semigroup. (This has be  mentioned in (a3)).

\hskip 18pt From Theorems \ref{cor-0423-demcomposition}-\ref{210320-yb-thm-main-explanations}, as well as the example in Section \ref{example}, we can see what follows: First,
 when $t=0$, $\mathfrak R_N(t)$ has a better smoothing effect than that of $\mathcal P_N(t)$, but when $t>0$, this is reversed. (The reason why
  the smoothing effect of $\mathfrak R_N(t)$ (with $t>0$) is weaker than that of $\mathcal P_N(t)$
 is due to wave-like terms contained in $\mathfrak R_N(t)$.)
  Second, when $t>0$,
  both $\mathfrak R_N(t)$ and $\mathcal P_N(t)$ have better smoothing effects than that of $\mathcal W_N(t)$, thus as $t>0$, the singularity of the flow is dominated by the wave-like component and the flow shows its wave-like nature.

	
	\item[$\bullet$] Frequencies. In Fourier analysis, the smoothing effect of a pseudo-differential operator corresponds to the growth of its symbol at high frequencies. The situation here is similar---Theorem \ref{cor-0423-demcomposition} and Theorem \ref{thm-ODE-meomery-asymptotic} corresponds to
each other and the smoothing effects of the components in \eqref{0423-demcomposition-eq}
 correspond to respectively the growths (in large $\eta$) of the  components in \eqref{thm-ODE-meomery-asymptotic-estimate}.
  From this and
   \eqref{thm-ODE-meomery-asymptotic-estimate}, we can see
   that when $t=0$, the last component in \eqref{thm-ODE-meomery-asymptotic-estimate} (It corresponds to $\mathfrak R_N$.) has a faster decay (in large $\eta$)  than  the first two components in \eqref{thm-ODE-meomery-asymptotic-estimate}, while when $t>0$, both the last one and the first one have faster decay than the second one.

	
	\item[$\bullet$] Non-triviality. Since the component $\mathfrak R_N$
has a hybrid heat/wave structure as mentioned above (In particular, we would like to mention
that it is not a finite dimensional low frequency operator.)  it is not
 negligible.
\end{itemize}


\item [(a8)]  We require $N\geq2$ in Theorem \ref{cor-0423-demcomposition}
since  $\mathcal W_1(\cdot) \equiv 0$ (when $N=1$).

\item[(a9)] Notice that both $\mathcal P_N(0)$ and $\mathcal W_N(0)$ are not projection operators in general and that
\begin{align*}
y_0 =  \mathcal P_N(0)y_0  +  \mathcal W_N(0) y_0\;\;\mbox{for each}\;\; y_0 \in L^2(\Omega)
~~\text{and}~~
\mathfrak R_N(0)=0.
\end{align*}
This is  further discussed in Proposition \ref{prop-PW-nonprojections}.

\item[(a10)] All the results of Theorems \ref{cor-0423-demcomposition}-\ref{210320-yb-thm-main-explanations}, except
\eqref{2021-april-hl-pl-nontrivial},
\eqref{yb-202104-wn-propagation} and
\eqref{202104-varphi-wn-y0},
 hold under the weaker  assumption $M\in C^\infty(\overline{ \mathbb R^+ })$. In Theorem \ref{thm-decomposition-extension} below we analyse the case of kernels $M$ in $C^{N_0}(\overline{ \mathbb R^+ })$, with $N_0\geq 2$.
The assumption  ($\mathfrak C$)  ensures \eqref{2021-april-hl-pl-nontrivial},
\eqref{yb-202104-wn-propagation} and
\eqref{202104-varphi-wn-y0}. Whether the same holds under weaker conditions on the kernel is an open problem, see Section \ref{sec-open-problems}.

\item[(a11)] The  decomposition \eqref{0423-demcomposition-eq}  has  applications in control theory. It allows, in particular, to compare the reachable sets for the controlled
 heat equations with and without memory term. We refer to \cite{WZZ-2} for a complete analysis of this issue.

\item[(a12)] There is  a large body of literature on the large time dynamics of
memory like problems (see, for instance, \cite{Dafermos,Farbizio-Gogi-Pata}) which is surely an important direction. Unfortunately, we are not able to use our decomposition to get such results.

\end{itemize}

\end{remark}

\subsection{Plan of the paper}
The rest of the paper is organized as follows. In Section 2 we analyze  the flow kernel $K_M$. In Section 3 we present a decomposition for  solutions to the ODE \eqref{ode-memory2-15}. Section 4 contains the proofs
of  Theorems \ref{cor-0423-demcomposition}-\ref{210320-yb-thm-flow-L2H4},
   and provides some other properties of the flow. In Section 5 we discuss, as an example, the case of the kernel $ M(t)=\alpha e^{\lambda t}$.  Section 6 lists several open problems. Section 7 contains an appendix.

\section{Properties of the flow kernel }

In this section we present some properties of the flow kernel $K_M$ in \eqref{new-kernel-KM}, which will be used later.

 In what follows, the space $C^k([a,b])$ (with $k\in \mathbb N^+$ and  $a<b$) is endowed with the  norm:
\begin{eqnarray*}
 \|f\|_{C^k([a,b])} := \sum_{l=0}^k  \Big\| \frac{d^l f}{dx^l} \Big\|_{C([a,b])},~~f\in C^k([a,b]).
\end{eqnarray*}

The following result provides basic estimates on iterated convolutions that will be used in the proof of Proposition \ref{prop-KM-Ck-regularity}. Its proof is put in the appendix.

\begin{lemma}\label{lem-convolution-estimates}
	Let $j,m\in \mathbb N^+$. Then for each sequence  $\{M_l\}_{l=1}^j\subset C^m(\overline{\mathbb R^+})$,   $M_1 * \cdots * M_j$ is in the space $C^{m+j-1}(\overline{\mathbb R^+})$
and satisfies that for each $k\in \{0,1,\ldots,m+j-1\}$,	
	\begin{align}\label{convolution-two-ineqs}
	\Big| \frac{d^k}{dt^k} M_1 * \cdots * M_j (t) \Big|
	\leq  \left( \sum_{l=\max\{0,j-1-k\}}^{j-1}  \frac{t^l}{l!}
	\right)
	\displaystyle\prod_{l=1}^j \| M_l \|_{C^{p}([0,t])},~~
		t> 0,
	\end{align}
	where $p:=\chi_{\mathbb N }(k-1)\big[(k-j)\chi_{\mathbb N }(k-j)+1\big]$ and $\chi_{\mathbb N }$ is the characteristic function of the set $\mathbb N $.
\end{lemma}

The following Proposition
 \ref{prop-KM-Ck-regularity}  provides estimates on the derivatives of the flow kernel $K_M$ that will serve for the proof of Theorem \ref{thm-ODE-meomery-asymptotic}, which is one of the
 tools  in the proof of the decomposition  in
 Theorem \ref{cor-0423-demcomposition}.

\begin{proposition}\label{prop-KM-Ck-regularity}
		The flow kernel $K_M\in C^\infty(S_+)$  (where $S_+$ is given in \eqref{new-kernel-KM})
satisfies that for 	each $\alpha,\beta\in \mathbb N$,
	\begin{align}\label{zhang-KM-deriavties-estimates}
	\big| \partial_t^\alpha \partial_s^\beta K_M(t,s) \big|
	\leq  e^{t-s} \bigg[
	\exp\Big( \beta \big( 1+|s| \big) \|M\|_{ C^{\alpha+\beta }([0,t-s]) } \Big)
	-1 \bigg],~~
		t>s.
	\end{align}
	\end{proposition}

\begin{proof}
	First of all, it follows from (\ref{new-kernel-KM}) that
	\begin{align}\label{KM-Mj-repsentation}
	K_M(t,s) = \sum_{j=1}^{+\infty}  \mathcal M_j(t,s),~~t\geq s,
	\end{align}
	where
	\begin{align}\label{def-Mj}
	\mathcal M_j(t,s) := \frac{(-s)^j}{j!} \underset{j}{ \underbrace{M * \cdots*  M} } (t-s),~~t\geq s.
	\end{align}

Next, we prove  $K_M\in C^\infty( S_+)$ showing
the convergence in $C^\infty(S_+)$ of the series on the right-hand side of  \eqref{KM-Mj-repsentation}.

To this end, we will estimate $\mathcal M_j$
	with $j\in\mathbb N^+$:  By the assumption $(\mathfrak C)$ and Lemma \ref{lem-convolution-estimates},
 we see that $\underset{j}{ \underbrace{M * \cdots*  M} }$ belongs to $C^\infty(\overline{\mathbb R^+})$
 and satisfies, for each  $k\in \mathbb N$,
\begin{align}\label{M-j-diff-estimate-ok}
	\Big| \frac{ d^{k} }{ d\tau^{k} } \underset{j}{ \underbrace{M * \cdots*  M} }  (\tau)	\Big|
	\leq e^{\tau} \|M\|_{ C^{k}([0,\tau]) }^j,~~
\tau>0.
\end{align}
 From \eqref{def-Mj} it follows that $\mathcal M_j \in C^\infty(S_+)$.

By direct computations,  for $\alpha,\beta\in \mathbb N$, we have
		$t\geq s$,
	\begin{align}\label{Mj-alpha-beta-estimate}
	\partial_t^\alpha \partial_s^\beta \mathcal M_j(t,s)
	=& \partial_s^\beta \partial_t^\alpha \mathcal M_j(t,s)
\nonumber\\
	=&	\frac{d^\beta}{ds^\beta} \bigg[
	\frac{(-s)^j}{j!} \Big( \frac{d^\alpha}{d\tau^\alpha} \underset{j}{ \underbrace{M * \cdots*  M} }(\tau) \Big)\big|_{\tau=t-s}
	\bigg]
	\nonumber\\
	=& \sum_{m=0}^\beta C_\beta^m   \frac{d^m}{d\tau^m}\Big(  \frac{(-s)^j}{j!} \Big)
	(-1)^{\beta-m} \Big( \frac{d^{\beta-m}}{d\tau^{\beta-m}} \frac{d^\alpha}{d\tau^\alpha} \underset{j}{ \underbrace{M * \cdots*  M} }(\tau) \Big)\big|_{\tau=t-s}
	\nonumber\\
	=&   \frac{ (-1)^{\beta} }{j!} \sum_{m=0}^{ \min\{\beta,j\} }
	C_\beta^m   (C_j^m m!) (-s)^{j-m}
	\Big( \frac{ d^{\alpha+\beta-m} }{ d\tau^{\alpha+\beta-m} } \underset{j}{ \underbrace{M * \cdots*  M} }(\tau) \Big)\big|_{\tau=t-s}.
	\end{align}
Here and in what follows, we
use the conventional notation  $0^0:=1$.
By
 (\ref{Mj-alpha-beta-estimate}) and \eqref{M-j-diff-estimate-ok}, one has that, when $t> s$,
	\begin{align}\label{Mj-alpha-beta-estimate-1}
	 \Big| \partial_t^\alpha \partial_s^\beta \mathcal M_j(t,s) \Big|
	\leq& \frac{1}{j!} \sum_{m=0}^{ j }
	(C_\beta^m m!) C_j^m   |s|^{j-m}
\Big( 	e^{t-s} \|M\|_{ C^{\alpha+\beta }([0,t-s]) }^j
\Big)
	\nonumber\\
	\leq& \frac{1}{j!} \beta^j
	\left( \sum_{m=0}^{ j }
	C_j^m    |s|^{j-m}   \right)
	e^{t-s} \|M\|_{ C^{\alpha+\beta }([0,t-s]) }^j
\nonumber\\
	=& \frac{ e^{t-s} }{j!}   \bigg(  \beta (1+|s|) \|M\|_{ C^{\alpha+\beta }([0,t-s]) } \bigg)^j.
	\end{align}
		Now,  by (\ref{Mj-alpha-beta-estimate-1}) it follows that
		the series in (\ref{KM-Mj-repsentation}) converges in
		$C^\infty(S_+)$.

Finally,  by (\ref{KM-Mj-repsentation}) and (\ref{Mj-alpha-beta-estimate-1}), after  direct computations, we see that  when
		$t>s$,
	\begin{align*}
	\Big| \partial_t^\alpha \partial_s^\beta K_M(t,s) \Big|
	\leq& \sum_{j=1}^\infty  \Big| \partial_t^\alpha \partial_s^\beta \mathcal M_j(t,s) \Big|
	\leq e^{t-s} \bigg[
	\exp\Big( \beta(1+|s| ) \|M\|_{ C^{\alpha+\beta }([0,t-s]) } \Big)
	-1 \bigg].
	\end{align*}
	This gives the desired estimate \eqref{zhang-KM-deriavties-estimates}
and ends the proof of Proposition \ref{prop-KM-Ck-regularity}.	
\end{proof}

The next Proposition \ref{prop-KM-analytic} concerns the  analyticity of the flow kernel $K_M$.
It will be used in the proofs of Proposition \ref{prop-KM-strong-unique}  and  Proposition
\ref{cor-evolution-high-regularity}.

\begin{proposition}\label{prop-KM-analytic}
	The flow kernel  $K_M$
	is real analytic on $S_+$
(where $S_+$ is given in \eqref{new-kernel-KM}).

\end{proposition}

\begin{proof}
	It suffices to prove that $K_M$ is real analytic over $S_T$ for each $T>0$, where
	\begin{align*}
	  S_T := \big\{ (t,s)\in \mathbb R^2~:~0\leq t-s \leq T \big\}.
	\end{align*}	
	To this end, we  fix an arbitrary $T>0$. Due to the analyticity of $M$ over $\overline{\mathbb R^+}$, there is a  domain $\widetilde{\mathcal O_T}$
of the complex plane $\mathbb C$,
with  $[0,T]\subset \widetilde{ \mathcal O_T }\subset\mathbb C$, so that
	$M$ has a unique analytic extension $\widetilde M$ to $\widetilde{\mathcal O_T}$.
	Moreover, we can take  a bounded and convex subdomain $ \mathcal O_T $ so that
	$[0,T]\subset \mathcal O_T\subset\subset \widetilde{ \mathcal O_T}$.
	The convolution  $*$ can then be extended to $\tilde{*}$
	in the following manner for $f,g\in C(\mathcal O_T;\mathbb C)$,
	\begin{eqnarray*}
		f \tilde{*} g(z) := \int_0^1 f((1-s)z)g(sz) z ds
		,~z\in \mathcal O_T.
	\end{eqnarray*}
	We now claim the following two properties:
	\begin{itemize}
		\item[]
(P1) For each $j\in\mathbb N^+$, $\underset{j}{ \underbrace{\widetilde M \tilde* \cdots \tilde* \widetilde M} }$ is an analytic extension of  $\underset{j}{ \underbrace{ M * \cdots *  M} }$ over $\mathcal O_T$;
		\item[] (P2) There is $C>0$ so that
		\begin{eqnarray*}
			\sup_{z\in \mathcal O_T}
			| \underset{j}{ \underbrace{\widetilde M \tilde* \cdots \tilde* \widetilde M} }(z) |
			\leq C^j\;\;\mbox{for all}\;\;j\in\mathbb N^+.
		\end{eqnarray*}
	\end{itemize}
	Indeed, one has
$$
	f\tilde{*}g|_{[0,T]}=f|_{[0,T]}* g|_{[0,T]}, \;\;\mbox{when}\;\;f,g\in C(\mathcal O_T;\mathbb C),
	$$
Here, $f\tilde{*}g|_{[0,T]}$,
$f|_{[0,T]}$ and $g|_{[0,T]}$
are respectively the restrictions of
 $f\tilde{*}g$, $f$ and $g$ over $[0,T]$.
Then, property (P1) follows from  the analyticity of $M$ at once, while the property (P2) can be proved by direct computations.
	
	Define the following subset of $\mathbb C^2$:
	\begin{align*}
	\mathcal D_T := \Big\{(t,s)\in \mathbb C^2
	~:~
	t-s \in \mathcal O_T \Big\}.
	\end{align*}
	It is clear that  $S_T\subset\mathcal D_T$. According to (P2) above, the following series
	uniformly converges over each compact subset of $\mathcal D_T$:
	\begin{align*}
	\sum_{j\geq 1} \frac{ (-s)^j }{ j! }  \underset{j}{ \underbrace{\widetilde M \tilde* \cdots \tilde* \widetilde M} } (t-s),~~(t,s) \in \mathcal D_T.
	\end{align*}
	Meanwhile, by (P1), we find that each term in the above series is analytic over $\mathcal D_T$.
	Hence,  the sum of this series is analytic over $\mathcal D_T$. From this and (\ref{new-kernel-KM}), we see
 that $K_M|_{S_T}$ (the restriction of $K_M$ over $S_T$) can be analytically extended  to $\mathcal D_T$. Therefore, it is real analytic over $S_T$. This ends the proof of Proposition \ref{prop-KM-analytic}.	
\end{proof}

The next Proposition \ref{prop-KM-strong-unique}  can be interpreted as a strong unique continuation  property or non-degeneracy of the kernel $K_M$, that  will be used in the proof of  Theorem \ref{thm-ODE-meomery-asymptotic}.

\begin{proposition}\label{prop-KM-strong-unique}
For each   $(t_0,s_0,v_1,v_2) \in 	(S_+ \times \mathbb S^1) \setminus I$, where	
\begin{align*}
I:= \Big\{ (t,t,\tau,\tau) ~:~ t\in \mathbb R,~\tau = \pm  1/\sqrt{2}  \Big\}
\cup
\Big\{ (t,0,\tau,0) ~:~ t \geq 0,~\tau=\pm 1 \Big\},
\end{align*}
with $\partial_{\vec{v}}:=v_1\partial_t+v_2\partial_s$ where  $\vec{v}:=(v_1,v_2)$, it holds that
\begin{align}\label{strong-v-l-zero}
\partial_{ \vec{v} }^l K_M(t_0,s_0)\neq0
\;\;\mbox{for some}\;\;
l\in \mathbb N.
\end{align}
\end{proposition}

\begin{remark}
	{\it Obviously,  real analytic functions on $\mathbb R^2$ do not necessarily fulfill the non-degeneracy condition above. Indeed, polynomials, for instance,  can vanish along lines in $\mathbb{R}^2$.}	
\end{remark}

\begin{proof}[Proof of Proposition \ref{prop-KM-strong-unique}]
	It suffices to prove \eqref{strong-v-l-zero} in the case that $\vec{v}=(0,1)$,
as other cases can be proved in a very similar way. By contradiction, suppose that   \eqref{strong-v-l-zero} with  $\vec{v}=(0,1)$ fails, i.e.,
	\begin{align}\label{KM-s-l-zero}
	\partial_{(0,1)}^l K_M(t_0,s_0)=0\;\;\mbox{for all}\;\; l\in \mathbb N.
	\end{align}
	Define
	\begin{align}\label{2.22,7.23}
	f(\lambda):= K_M(t_0,t_0-\lambda),~~\lambda\geq0.
		\end{align}
	Two facts on $f$ are given as follows: First, by the real analyticity of $M$ over $\overline{\mathbb R^+}$, we see from Proposition \ref{prop-KM-analytic} that $K_M$ is real analytic over $S_+$.
 This, along with \eqref{2.22,7.23}, yields that $f$ is real analytic over  $\overline{\mathbb R^+}$.
	Second,
	by \eqref{KM-s-l-zero} and (\ref{2.22,7.23}), we find that  $f$ vanishes of infinite order at $\lambda=t_0 -s_0$.
	From these two facts, we see that
	$f\equiv 0$ over $\overline{\mathbb R^+}$,
	which, along with (\ref{2.22,7.23}), yields
	\begin{align*}
	K_M(t_0,t_0-\lambda)=0,~~\lambda\geq0.
		\end{align*}
	The above, together with  \eqref{new-kernel-KM}, shows
	\begin{align*}
	0= \sum_{j=1}^{+\infty}  \frac{ ( \lambda-t_0 )^j }{ j! } \underset{j}{ \underbrace{M*\cdots*M} }(\lambda),~~\lambda\geq0,
		\end{align*}
	which leads to
	\begin{align}\label{KM-s-l-zero-expression}
	0=  M(\lambda)  +
	\sum_{j=2}^{+\infty}  \frac{ ( \lambda-t_0 )^{j-1} }{ j! } \underset{j}{ \underbrace{M*\cdots*M} }(\lambda),~~\lambda\geq0.
		\end{align}
	
	Next, we arbitrarily fix $T>0$. For each $j\in \mathbb{N}^+\setminus\{1\}$, we define an operator $\mathcal K_j$ on
	$C([0,T])$ in the following manner: given $g\in C([0,T])$, set
	\begin{align}\label{KM-s-l-zero-expression-1}
	\mathcal K_j(g)(\lambda):= \frac{ ( \lambda-t_0 )^{j-1} }{ j! } \underset{j-1}{ \underbrace{M*\cdots*M} } * g(\lambda),~~
	0\leq \lambda \leq T,
	\end{align}
	which is well-defined, linear and bounded.
	Let
	$$
	\mathcal Q:= \sum_{j=2}^{+\infty} \mathcal K_j.
	$$
	One can directly check that $\mathcal Q\in \mathcal L(C([0,T]))$, the Banach space of all linear and bounded operators on $C([0,T])$.
	Thus,  we deduce from \eqref{KM-s-l-zero-expression} that
	\begin{align}\label{Id-Q-M=zero}
	(Id+ \mathcal Q) (M|_{[0,T]})  =   \Big( Id+ \sum_{j=2}^{+\infty} \mathcal K_j \Big) (M|_{[0,T]}) =0,
	\end{align}
	where $Id$ is the identity operator on $C([0,T])$.
	
	We now claim that $(Id+ \mathcal Q)^{-1}$
exists in $\mathcal L( C([0,T]) )$.
	 When this is done, we can use \eqref{Id-Q-M=zero}, \eqref{new-kernel-KM}
and the  arbitrariness of $T$ to see that  $M\equiv0$, which contradicts the assumption $(\mathfrak C)$.
Consequently,  \eqref{strong-v-l-zero} is true.
	
	The  remainder is to show the above claim. To this end, we arbitrarily fix $k\in \mathbb N^+$ and then estimate $\mathcal Q^k$
in the following manner:  Set
	\begin{align*}
	T_{t_0} := \sup_{ 0\leq \lambda \leq T}  | \lambda -t_0 |
	= \max \big\{ |t_0|,~ |T-t_0| \big\}.
	\end{align*}
	Then from  \eqref{KM-s-l-zero-expression-1},
	one can directly check that when $j_1,\ldots,j_k\geq 2$,
	\begin{align*}
	\| \mathcal K_{j_1} \cdots \mathcal K_{j_k} \|_{ \mathcal L( C([0,T]) ) }
	\leq&  \frac{1}{j_1!}  \cdots \frac{1}{j_k!}
	T_{t_0}^{j_1+\cdots +j_k -k}
	\big\| \underset{j_1+\cdots+j_k-k}{ \underbrace{|M|*\cdots*|M|} } \big\|_{ L^1( [0,T] ) }
	\nonumber\\
	\leq&  \frac{1}{j_1!}  \cdots \frac{1}{j_k!}
	T_{t_0}^{j_1+\cdots +j_k -k}
	\frac{ T^{j_1+\cdots +j_k-k } }{ (j_1+\cdots+j_k-k)! }
	\| M\|_{C([0,T])}^{j_1+\cdots +j_k -k}
	\nonumber\\
	\leq& \frac{1}{ k! }  \frac{ 1 }{j_1!}  \cdots \frac{1}{j_k!}
	\Big( T_{t_0} T \| M\|_{C([0,T])} \Big)^{j_1+\cdots +j_k -k }.
	\end{align*}
	This, along with the definition of $\mathcal Q$, yields
	\begin{align*}
	&\| \mathcal Q^k \|_{ \mathcal L( C([0,T]) ) }
	=  \Big\| \sum_{ j_1,\ldots,j_k \geq 2 }
	\mathcal K_{j_1} \cdots \mathcal K_{j_k} \Big\|_{ \mathcal L( C([0,T]) ) }
	\nonumber\\
	\leq& \frac{ 1 }{ k! }
	\prod_{m=1}^k \left( \sum_{j_m\geq 2} \frac{1}{j_m!}
	\Big( T_{t_0} T \| M\|_{C([0,T])} \Big)^{j_m-1}
\right)
	\leq \frac{  1  }{ k! }
	\exp\Big(   k T_{t_0} T   \| M\|_{C([0,T])} \Big).
	\end{align*}
	So   $\displaystyle\sum_{k=0}^{+\infty} (-\mathcal Q)^k$  converges in  $\mathcal L( C([0,T]) )$.
	Then we have
	\begin{align*}
	Id= (Id + \mathcal Q ) \sum_{k=0}^{+\infty} (-\mathcal Q)^k.
	\end{align*}
	Therefore, $(Id+ \mathcal Q)^{-1}$  exists in $\mathcal L( C([0,T]) )$. This
	completes  the proof of  Proposition \ref{prop-KM-strong-unique}.		
\end{proof}

The following  Proposition  \ref{prop-KM-regularity} presents
a weighted estimate of
the flow kernel $K_M$.
 It will be used in the proof of
Proposition \ref{prop-varPhi-expression}
that provides an explicit expression of the gap between the heat evolution with and without memory.

\begin{proposition}\label{prop-KM-regularity}
	For each  $\lambda \in \mathbb R$,
		\begin{align}\label{KM-regularity-estimate}
		\int_{0}^t e^{-\lambda(t-s)}
		| K_M(t,s) | ds
		\leq
		\exp\bigg( t \int_0^t e^{-\lambda \tau} |M(\tau)| d\tau \bigg)
		- 1,
~~t\geq0.
				\end{align}
	
\end{proposition}

\begin{proof}
	Arbitrarily fix
 $\lambda\in \mathbb R$. Define the following weighted memory kernel:
	\begin{align}\label{def-M-lambda}
	M_\lambda(t) := e^{-\lambda t} M(t),~~t\geq 0.
	\end{align}
	
We claim
	\begin{align}\label{claim-KM-KMlambda-relation}
	e^{-\lambda (t-s)}  K_{M} (t,s)
		= \sum_{j=1}^{+\infty} \frac{ (-s)^j }{j!}
\underset{j}{ \underbrace{M_\lambda *\cdots* M_\lambda} } (t-s),
	~~(t,s)\in S_+.
	\end{align}
	Indeed,
\eqref{claim-KM-KMlambda-relation} follows from (\ref{new-kernel-KM}) and the following identity:
	\begin{align*}
	\underset{j}{ \underbrace{M_\lambda *\cdots* M_\lambda } }(\tau)
	= e^{-\lambda \tau} \underset{j}{ \underbrace{M *\cdots* M } }(\tau),~~\tau\geq 0,\; j\in \mathbb N^+,
	\end{align*}
	which can be verified directly.

 Next, we arbitrarily fix $t>0$. By  the iterative use of the  Young's inequality:
	\begin{align*}
	\|f*g\|_{L^1(0,t)}   \leq \|f\|_{L^1(0,t)}  \|g\|_{L^1(0,t)},\;\mbox{when}\; f,g \in L^1(0,t),
	\end{align*}
one has that
 \begin{align*}
 \| \underset{j}{ \underbrace{M_\lambda *\cdots* M_\lambda} } \|_{L^1(0,t)}
	\leq \|M_\lambda\|_{L^1(0,t)}^j,~ ~j\in \mathbb N^+.
	\end{align*}
 This, along with  \eqref{claim-KM-KMlambda-relation}, yields
	\begin{align*}
	\int_{0}^{t}
	\big| e^{-\lambda (t-s)}  K_{M} (t,s) \big| ds
	&\leq
	\sum_{j=1 }^{+\infty} \frac{ t^j }{ j! }
\|\underset{j}{ \underbrace{M_\lambda *\cdots* M_\lambda} }\|_{L^1(0,t)}
\leq
	\exp\Big( t \|M_\lambda\|_{L^1(0,t)} \Big)
	- 1.
	\end{align*}
Then \eqref{KM-regularity-estimate} follows from \eqref{def-M-lambda}.
	This concludes the proof of Proposition \ref{prop-KM-regularity}.
\end{proof}

\section{Parameterized ODEs with memory}
In this section we analyze the ODE \eqref{ode-memory2-15}, i.e.,
\begin{eqnarray}\label{ode-memory}
w^\prime(t)+\eta w(t) + \int_0^t M(t-s)w(s) ds=0,~~t>0;~~~
w(0)=1,
\end{eqnarray}
where
$\eta>0$ is a parameter.
First of all, by a standard method in the ODE theory, one can easily check that
the equation (\ref{ode-memory}) has a unique solution, denoted by $w_\eta$,
in the space $C^1(\overline{\mathbb R^+})$.
The main result of this section is the next Theorem \ref{thm-ODE-meomery-asymptotic} which gives
a decomposition in terms of $\eta$ for the solution $w_\eta$.
It  lays a solid foundation for the proof of
Theorem
 \ref{cor-0423-demcomposition}.

\begin{theorem}\label{thm-ODE-meomery-asymptotic}
For each integer $N\geq 2$, the solution $w_\eta$ (with  $\eta>0$) to the equation (\ref{ode-memory}) satisfies
\begin{align}\label{thm-ODE-meomery-asymptotic-estimate}
	w_\eta(t) ~=~  e^{-\eta t} \bigg( 1 + \sum_{l=0}^{N-1} p_{l}(t)  \eta^{-l-1}  \bigg)
	+   \sum_{l=0}^{N-1} h_{l}(t)  \eta^{-l-1}
	+   R_N(t,\eta) \eta^{-N-1},~~
	t \geq 0,
	\end{align}
where  $\{h_l\}_{l\in \mathbb{N}}$ and $\{p_l\}_{l\in\mathbb{N}}$ are given by (\ref{thm-ODE-meomery-asymptotic-estimate-hypobolic})
and $R_N$ is given by \eqref{0921-RN-good-remainder}.
In addition, the following conclusions are true:

\noindent
$(i)$ 	The function  $R_N$ is in the space
	$C\big(\overline{\mathbb R^+} \times \overline{\mathbb R^+} \big)
	\cap C\big(\mathbb R^+; C(\mathbb R^+)  \big)$ and fulfills the  estimate:
	\begin{eqnarray}\label{estimate-p-q-R}
	\| R_N(t,\cdot) \|_{ C(\mathbb R^+) }
	\leq
    e^t \bigg\{
	\exp\bigg[ N(1+t) \bigg(\sum_{j=0}^N \max_{0\leq s\leq t} \Big| \frac{d^j}{ds^j}M(s) \Big| \bigg) \bigg]  - 1  \bigg\},~~
	t\geq 0.
	\end{eqnarray}

\noindent $(ii)$ For each $t\in \overline{\mathbb R^+}$, neither  $\{h_l(t)\}_{l\geq 1}$ nor $\{p_l(t)\}_{l\in\mathbb{N}}$
	is the  null sequence.	
\end{theorem}

\begin{remark}
	{\it
  $(i)$ The decomposition (\ref{thm-ODE-meomery-asymptotic-estimate}) serves for that in  \eqref{0423-demcomposition-eq}.
 The components  in  the decompositions (\ref{thm-ODE-meomery-asymptotic-estimate}) and  \eqref{0423-demcomposition-eq}
 correspond to each other.

$(ii)$ The conclusion $(ii)$ in Theorem \ref{thm-ODE-meomery-asymptotic} corresponds to
\eqref{2021-april-hl-pl-nontrivial} in
Theorem \ref{cor-0423-demcomposition} discussed in
Remark \ref{remark1.2-2-13}.}
\end{remark}

To prove Theorem \ref{thm-ODE-meomery-asymptotic}, we need  several lemmas.
The first one refers to the representation of  the solutions of (\ref{ode-memory}) in terms of the flow kernel $K_M$ (given by (\ref{new-kernel-KM})).

\begin{lemma}\label{lem-ode-memory-infinite-series}
	The solution $w_\eta$ (with  $\eta\in \mathbb{R}$) to (\ref{ode-memory}) satisfies
	\begin{eqnarray}\label{infinite-series-express}
	w_\eta(t) =   e^{-\eta t} +
	\int_0^t K_M(t,s)  e^{-\eta s }  ds,~~
	t\geq0.
	\end{eqnarray}
	\end{lemma}
\begin{proof}
	Fix an arbitrary $\eta \in \mathbb R$.
	From (\ref{ode-memory}), it follows that
	\begin{eqnarray*}
		w_\eta^\prime(t) + \eta w_\eta(t)
		=  - M*w_\eta(t),~t\geq 0;
		~~w_\eta(0)=1.
	\end{eqnarray*}
	The above yields
	\begin{eqnarray}\label{w-eta-integral-eq-1906}
	w_\eta(t) =  e^{-\eta t} -  \int_0^t e^{-\eta(t-s)} (M*w_\eta) (s) ds
	=  e^{-\eta t}  - (e^{-\eta\cdot}*M*w_\eta)(t),\; t\geq 0.
	\end{eqnarray}
	Then we arbitrarily fix $T>0$ and define
the following operator:
	\begin{eqnarray}\label{def-QT}
	\mathcal Q_T (f) :=
	e^{-\eta\cdot}*M*f
	\;\;\mbox{for each}\;\;
	f\in C([0,T]).
	\end{eqnarray}
	One can easily check that $\mathcal Q_T$ is a linear and bounded operator on $C([0,T])$.
	By \eqref{def-QT}, we see that  for each $k\in \mathbb N^+$,
	\begin{align}\label{QT-k-expression}
	\mathcal Q_T^k(f)=\underset{k}{ \underbrace{ (e^{-\eta\cdot}*M) *\cdots* (e^{-\eta\cdot}*M) } } * f,
	~~f\in C([0,T]).
	\end{align}
	From \eqref{QT-k-expression}, one can directly check
	\begin{align*}
	\| \mathcal Q_T^k \|_{\mathcal L(C([0,T]))}
	\leq&  \| \underset{k}{ \underbrace{ (e^{-\eta\cdot}*M) *\cdots* (e^{-\eta\cdot}*M) } } \|_{L^1(0,T)}
	\nonumber\\
	\leq&  \| e^{-\eta\cdot}*M \|_{C([0,T])}^k
	\bigg( \int_0^T \int_0^{t_1} \cdots \int_0^{t_{k-1}}  dt_k \cdots dt_1
	\bigg)
	\nonumber\\
	\leq& \Big( e^{ T|\eta| }\|M\|_{L^1(0,T)}  \Big)^k
	\frac{T^k}{k!}.
	\end{align*}
	From this, we see that the series 	$\sum_{k=0}^\infty (-\mathcal Q_T)^k$ converges in $\mathcal L\big( C([0,T]) \big)$
and that
	\begin{eqnarray}\label{3.9,7.23}
	(1+\mathcal Q_T)^{-1} =  \sum_{k=0}^\infty (-\mathcal Q_T)^k.
	\end{eqnarray}
	Meanwhile, it follows from  (\ref{w-eta-integral-eq-1906}) and (\ref{def-QT}) that
	\begin{eqnarray}\label{3.9,2-16}
		w_\eta |_{[0,T]}= -\mathcal Q_T \Big( w_\eta|_{[0,T]} \Big)
		+ e^{-\eta \cdot} |_{[0,T]}.
	\end{eqnarray}
	Here, $w_\eta |_{[0,T]}$ and $e^{-\eta \cdot}|_{[0,T]}$
denote respectively the restrictions of $w_\eta$ and $e^{-\eta \cdot}$
over $[0,T]$.

Now, by \eqref{3.9,2-16} and  (\ref{3.9,7.23}), we find
	\begin{eqnarray*}\label{1906-w-QT}
		w_\eta(t)=  \Bigg(\Big(\sum_{j=0}^\infty (-\mathcal Q_T)^j\Big) \Big( e^{-\eta \cdot}|_{[0,T]} \Big)\Bigg)(t),\; t\in[0,T],
	\end{eqnarray*}
	from which and (\ref{QT-k-expression}), it follows  that
	\begin{eqnarray*}
		w_\eta(t) &=&  e^{-\eta t} + \sum_{j=1}^\infty  (-1)^j \underset{j}{ \underbrace{(e^{-\eta\cdot}*M)*\cdots*(e^{-\eta\cdot}*M)} }*e^{-\eta\cdot}(t)
		\nonumber\\
		&=& e^{-\eta t} +  \sum_{j=1}^\infty  (-1)^j
		\underset{j+1}{ \underbrace{e^{-\eta\cdot}*\cdots*e^{-\eta\cdot}} }
		*  \underset{j}{ \underbrace{M*\cdots*M} }(t),\; t\in[0,T].
	\end{eqnarray*}
	This, together with   the following equality:
	\begin{eqnarray*}
		\underset{j+1}{ \underbrace{e^{-\eta\cdot}*\cdots*e^{-\eta\cdot}} } (t)
		= e^{-\eta t} t^j/j!,~~t\geq0,~~j\in\mathbb N^+,
	\end{eqnarray*}
	shows
	\begin{eqnarray*}
		w_\eta(t) = e^{-\eta t} +  \sum_{j=1}^\infty   \int_0^t e^{-\eta s}
		\frac{ (-s)^j}{j!}
		\underset{j}{ \underbrace{M*\cdots*M} }(t-s)ds,\; t\in [0,T].
	\end{eqnarray*}
	Since $T>0$ was arbitrarily taken, the above, along with (\ref{new-kernel-KM}),
	leads to  (\ref{infinite-series-express}). This ends the proof of Lemma \ref{lem-ode-memory-infinite-series}.		
\end{proof}

The next lemma  will be used to get an asymptotic expansion of the
second term   on the right-hand side of (\ref{infinite-series-express}).

\begin{lemma}\label{lem-estimate-for-convolution}
	Given an integer  $N\geq 2$ and a number $\eta>0$, the following equality is true:
	\begin{align}\label{lem-estimate-for-convolution-ineq}
	\int_0^t e^{-\eta s}  K_M(t,s)  ds
	~=~&        e^{-\eta t} \sum_{l=0}^{N-1}
	\Big( -\partial_s^l K_M(t,s) \big|_{s=t}  \Big) \eta^{-l-1}
	+     \sum_{l=0}^{N-1} \Big( \partial_s^l K_M(t,s) \big|_{s=0}  \Big)
	\eta^{-l-1}
	\nonumber\\
	& \quad +  \eta^{-N}
	\int_0^t e^{-\eta s} \partial_s^N K_M(t,s)  ds,\;\;t\geq 0.
	\end{align}	
\end{lemma}

\begin{proof}
	Fix $N\geq 2$, $\eta>0$ and $t>0$ arbitrarily. Given $g\in C([0,t])$, we define
	\begin{eqnarray}\label{def-mathcalP-g}
	\mathcal F_\eta (g)  := \int_0^t e^{-\eta s} g(s) ds.
	\end{eqnarray}
	We first claim
	\begin{eqnarray}\label{f-iterative-k-estimate}
	\mathcal F_\eta (g)
	&=&   \eta^{-N} \mathcal F_\eta(g^{(N)})
	+  \sum_{l=0}^{N-1} \eta^{-l-1}
	\Big(  g^{(l)}(0)  -   e^{-\eta t} g^{(l)}(t)  \Big),\;\;\mbox{when}\;\;g\in C^N([0,t]).~~~
	\end{eqnarray}
	Given  $g\in C^N([0,t])$, from (\ref{def-mathcalP-g}) and  using the integration by parts, we find
	\begin{eqnarray*}
		\mathcal F_\eta (g)  &=&  (-\eta)^{-1}  \int_0^t \frac{d}{ds} e^{-\eta s} g(s) ds
		\nonumber\\
		&=&  (-\eta)^{-1} \Big( e^{-\eta s} g(s) \Big) \Big|_{s=0}^t   +
		\frac{1}{\eta}  \int_0^t e^{-\eta s} g^\prime(s) ds
		\nonumber\\
		&=&  \eta^{-1}  \Big( g(0) -  e^{-\eta t} g(t)  \Big)
		+ \eta^{-1} \mathcal F_\eta(g^\prime).
	\end{eqnarray*}
	Now, the iterative use of the above equality (to the derivatives of $g$) gives (\ref{f-iterative-k-estimate}).

	Next, it follows by Proposition \ref{prop-KM-Ck-regularity}
	that $K_M(t,\cdot)\in C^\infty([0,t])$. Thus we can apply  (\ref{f-iterative-k-estimate}) (where $g(s)=K_M(t,s)$, $0\leq s\leq t$) to get
	\begin{eqnarray*}
		\int_0^t K_M(t,s) e^{-\eta s} ds
		&=& \eta^{-N}
		\int_0^t e^{-\eta s }
		\partial_s^N K_M(t,s) ds
		\\
		& &  +  \sum_{l=0}^{N-1} \eta^{-l-1} \bigg(
		\partial_s^l K_M(t,s)\big|_{s=0}
		- e^{-\eta t}
		\partial_s^l K_M(t,s) \big|_{s=t}
		\bigg).
		\nonumber
	\end{eqnarray*}
	This leads to (\ref{lem-estimate-for-convolution-ineq}) and completes the proof of Lemma \ref{lem-estimate-for-convolution}. 		
\end{proof}

The following Lemma \ref{lem-KM-boundary-value} will be used in the proof of  $(ii)$ in
 Theorem \ref{thm-ODE-meomery-asymptotic}.

\begin{lemma}\label{lem-KM-boundary-value}
	Let $\{h_l\}_{l\in \mathbb{N}}$ and $\{p_l\}_{l\in\mathbb{N}}$ be given by (\ref{thm-ODE-meomery-asymptotic-estimate-hypobolic}). Then for each $l\in \mathbb N$,
	\begin{align}\label{KM-Ck-equality}
	\partial_s^l K_M(t,s)|_{s=0} =  h_l(t)
	~~\mbox{and}~~
	\partial_s^l K_M(t,s)|_{s=t} = - p_l(t),~~t\geq 0.
	\end{align}
	\end{lemma}

\begin{proof}
Recall the conventional notation:  $0^0=1$.
	First of all, we recall that  $\mathcal M_j$ is given in \eqref{def-Mj}.
	Given   $t\geq 0$, by \eqref{Mj-alpha-beta-estimate}, where $(\alpha,\beta)=(0,l)$, it follows  that
	\begin{align*}
	\partial_s^l \mathcal M_j(t,s) |_{s=0}
	=&  \frac{ (-1)^{l} }{j!} \sum_{m=0}^{ \min\{l,j\} }
	C_l^m   C_j^m m! \Big( (-s)^{j-m} \big|_{s=0} \Big)
	\Big( \frac{ d^{l-m} }{ d\tau^{l-m} } \underset{j}{ \underbrace{M * \cdots*  M} } \Big) (t)
	\nonumber\\
	=& \frac{ (-1)^{l} }{j!}
	\bigg( \chi_{[0,l]} (j) C_l^j j!   \bigg)
	\frac{ d^{l-j} }{ dt^{l-j} }  \underset{j}{ \underbrace{M * \cdots*  M} } (t),
	\end{align*}
	which, along with \eqref{KM-Mj-repsentation}, yields
	\begin{align*}
	\partial_s^l K_M(t,s) |_{s=0}
	= \sum_{j=1}^{+\infty} \partial_s^l \mathcal M_j(t,s) |_{s=0}
	= (-1)^l \sum_{j=0}^l  C_l^j
	\frac{ d^{l-j} }{ dt^{l-j} }  \underset{j}{ \underbrace{M * \cdots*  M} } (t).
	\end{align*}
	This, along with \eqref{thm-ODE-meomery-asymptotic-estimate-hypobolic}, leads to the first equality in \eqref{KM-Ck-equality}.
	
	Next, from \eqref{Mj-alpha-beta-estimate} with $(\alpha,\beta)=(0,l)$, it follows that
	\begin{align}\label{Mj-Cl-s=t}
	\partial_s^l \mathcal M_j(t,s) |_{s=t}
	=&  \frac{ (-1)^{l} }{j!} \sum_{m=0}^{ \min\{l,j\} }
	C_l^m   C_j^m m! (-t)^{j-m}
	\Big( \frac{ d^{l-m} }{ d\tau^{l-m} } \underset{j}{ \underbrace{M * \cdots*  M} } \Big) (0).
	\end{align}
	Meanwhile, by Lemma \ref{lem-convolution-estimates}, we find
	\begin{align*}
	\frac{ d^{l-m} }{ d\tau^{l-m} } \underset{j}{ \underbrace{M * \cdots*  M} } (0)
	=0,\;\;\mbox{when}\;\;l-m< j-1.
	\end{align*}
	From the above and (\ref{Mj-Cl-s=t}), we see
\begin{align*}
	\partial_s^l \mathcal M_j(t,s) |_{s=t}
	=&   (-1)^{l}  \sum_{m=0}^{ \min\{l,j\} }
	C_l^m   \frac{ (-t)^{j-m} } { (j-m)! }
    \chi_{[0,l-m]}(j-1)
	\frac{ d^{l-m} }{ d\tau^{l-m} } \underset{j}{ \underbrace{M * \cdots*  M} }  (0).
	\end{align*}
	This, along with \eqref{KM-Mj-repsentation}, yields
	\begin{align*}
	 \partial_s^l K_M(t,s) |_{s=t}
	=& \sum_{j=1}^{+\infty} \partial_s^l \mathcal M_j(t,s) |_{s=t}
	\nonumber\\
	=&  (-1)^l \sum_{j=1}^{l+1}
	\sum_{m=0}^{ \min\{l,j,l-j+1\} }
	\frac{ (-t)^{j-m} } { (j-m)! }
	\Big(  C_l^{l-m}
	\frac{ d^{l-m} }{ d\tau^{l-m} } \underset{j}{ \underbrace{M * \cdots*  M} }  (0)
	\Big).
	\end{align*}
	Replacing  $j-m$  by a new variable $q$ in the above, using \eqref{thm-ODE-meomery-asymptotic-estimate-hypobolic}, we obtain the second equality in \eqref{KM-Ck-equality}.
	
	Hence, we finish the proof of Lemma \ref{lem-KM-boundary-value}. 		
\end{proof}

We now are on the  position to prove Theorem \ref{thm-ODE-meomery-asymptotic}.

\begin{proof}[Proof of Theorem \ref{thm-ODE-meomery-asymptotic}]
	Fix $N\geq 2$ and $\eta>0$ arbitrarily. The proof is structured in three steps.

\vskip 5pt
\noindent{\it Step 1. Proof of \eqref{thm-ODE-meomery-asymptotic-estimate}.}
	
	\vskip 5pt
	By    Lemmas \ref{lem-ode-memory-infinite-series}, \ref{lem-estimate-for-convolution} and \ref{lem-KM-boundary-value}, and by  \eqref{0921-RN-good-remainder}, we find
	\begin{align*}
	w_\eta(t)
	=& e^{-\eta t} +e^{-\eta t} \sum_{l=0}^{N-1}
	\Big( -\partial_s^l K_M(t,s) \big|_{s=t}  \Big) \eta^{-l-1}
	+     \sum_{l=0}^{N-1} \Big( \partial_s^l K_M(t,s) \big|_{s=0}  \Big)
	\eta^{-l-1}
	\nonumber\\
	& \quad + \eta^{-N}  \int_0^t e^{- \eta s} \partial_{s}^N K_M(t,s)  ds
	\nonumber\\
	=& e^{-\eta t} \sum_{l=0}^{N-1}
	\Big( 1 +p_l(t) \eta^{-l-1} \Big)
	+     \sum_{l=0}^{N-1} h_l(t)  \eta^{-l-1}
	+ R_N(t,\eta) \eta^{-N-1},\; t\geq 0,
	\end{align*}
which leads to \eqref{thm-ODE-meomery-asymptotic-estimate}.

\vskip 5pt
\noindent{\it Step 2. Proof of conclusion $(i)$.}
	
	\vskip 5pt
First, we have
\begin{eqnarray}\label{3.13,11-18}
R_N\in C( \overline{\mathbb R^+} \times \overline{\mathbb R^+}).
\end{eqnarray}
Indeed, we apply Proposition \ref{prop-KM-Ck-regularity}
to see
\begin{align}\label{zhang-1027-1}
    K_M \in C^\infty(S_+).
\end{align}
From \eqref{0921-RN-good-remainder} and
 \eqref{zhang-1027-1}, \eqref{3.13,11-18} follows
 at once.

 Second, we have
	\begin{align}\label{RN-remainder-CL-regularity}
	R_N \in C( \mathbb R^+; C(\mathbb R^+) ).
	\end{align}
	Indeed, it follows from  \eqref{0921-RN-good-remainder} that when $t_2\geq t_1 > 0$,
	\begin{align*}
		& \| R_N(t_1,\cdot)  - R_N(t_2,\cdot) \|_{ C(\mathbb R^+) }
	\nonumber\\
	\leq&  \sup_{\tau>0}  \left[ \int_0^{t_1} \tau e^{-\tau t} | \partial_s^N K_M(t_1,s) -  \partial_s^N K_M(t_2,s) | ds
	+ \int_{t_1}^{t_2} \tau e^{-\tau t} | \partial_s^N K_M(t_2,s) | ds
	\right]
	\nonumber\\
	\leq&  \| \partial_s^N K_M(t_1,\cdot) -  \partial_s^N K_M(t_2,\cdot) \|_{ L^\infty(0,t_1) }  +
	\frac{1}{t_1}  \int_{t_1}^{t_2} | \partial_s^N K_M(t_2,s) | ds.
	\end{align*}
	This, along with \eqref{zhang-1027-1},
 yields (\ref{RN-remainder-CL-regularity}).

 Third,   from \eqref{0921-RN-good-remainder} and Proposition \ref{prop-KM-Ck-regularity} (with $(\alpha,\beta)=(0,N)$), we see
	\begin{align*}
	\|R_N(t,\cdot)\|_{ C(\mathbb R^+) }
	\leq& \Big( \sup_{\tau>0}   \int_0^t \tau e^{-\tau s} ds  \Big)
	\| \partial_s^N K_M(t,\cdot)\|_{L^\infty(0,t)}
	\nonumber\\
	\leq& e^t \bigg\{
	\exp\bigg[ N(1+t) \bigg(\sum_{j=0}^N \max_{0\leq s\leq t} \Big| \frac{d^j}{ds^j}M(s) \Big| \bigg) \bigg]  - 1  \bigg\},\; t\geq 0,
	\end{align*}
	which leads to \eqref{estimate-p-q-R}.

\vskip 5pt
\noindent{\it Step 3. Proof of conclusion $(ii)$.}
	
	\vskip 5pt
	
	 Fix any $t\geq 0$. We
apply  Proposition \ref{prop-KM-strong-unique} to $(t_0,s_0)=(t,0)$ and $\vec{v}=(0,1)$;
	$(t_0,s_0)=(t,t)$ and $\vec{v}=(0,1)$, respectively,
to find $j_1, j_2\in \mathbb{N}$ so that
	\begin{eqnarray}\label{3.15,7.23}
	\partial_s^{j_1} K_M (t,0) \neq 0\;\;\;
	\mbox{and}\;\;\; \partial_s^{j_2} K_M (t,t) \neq 0.
	\end{eqnarray}
	Meanwhile, by \eqref{thm-ODE-meomery-asymptotic-estimate-hypobolic} and \eqref{KM-Ck-equality}, we have
	\begin{eqnarray}\label{3.16,7.23}
	h_0\equiv0,~\partial_s^{j_1} K_M(t,0)=h_{j_1}(t)\;\;\;
	\mbox{and}\;\;\;
	\partial_s^{j_2} K_M(t,t)=-p_{j_2}(t).
	\end{eqnarray}
	Now, from (\ref{3.15,7.23}) and (\ref{3.16,7.23}), we see that
	neither $\{h_l(t)\}_{l\geq 1}$ nor $\{p_l(t)\}_{l\in\mathbb{N}}$ is the zero sequence.

This concludes  the proof of Theorem \ref{thm-ODE-meomery-asymptotic}.
\end{proof}

\section{Analysis of the memory-flow}

In this section we present several technical propositions describing the nature of each component of the decomposition
\eqref{0423-demcomposition-eq}.
 We then prove  Theorems \ref{cor-0423-demcomposition}-\ref{210320-yb-thm-flow-L2H4} one by one. At last we prove some complementary  properties of the flow and present an extension of  Theorems \ref{cor-0423-demcomposition}-\ref{210320-yb-thm-main-explanations}.

\subsection{On the components in the decomposition
 }\label{appendix-expansion}

The aim of this  subsection is to explain the meaning of each component
in the decomposition \eqref{0423-demcomposition-eq},  and discuss their regularity properties.
 This  will help us understanding the hybrid parabolic-hyperbolic behavior  of the flow more deeply.

The next Proposition \ref{pro-PN-heatlike} (which is the conclusion $(i)$ in Theorem \ref{210320-yb-thm-main-explanations})  assures the smoothing effect
of the first component  in \eqref{0423-demcomposition-eq}.
\begin{proposition}\label{pro-PN-heatlike}
Let $N\geq 2$ be an integer. Then   $\mathcal P_N(t) \mathcal H^{-\infty}\subset \mathcal H^{+\infty}$
for all $t>0$.
\end{proposition}
\begin{proof}
	This directly follows from the definition of $\mathcal P_N(\cdot)$ (in \eqref{def-PN-HN-RN}) and the smoothing effect of the heat semigroup $\{e^{tA}\}_{t\geq 0}$. 	
\end{proof}

The next Proposition \ref{pro-HN-wavelike} (which contains the conclusion $(ii)$ in Theorem \ref{210320-yb-thm-main-explanations}) shows  the propagation of singularities
along the time direction
for
the second component $\mathcal W_N(\cdot)$. This shows the hyperbolic nature of $\mathcal W_N(\cdot)$,  with  null velocity of propagation.

\begin{proposition}\label{pro-HN-wavelike}
 Let $N\geq 2$ be an integer and let $y_0\in  \mathcal H^{-\infty}$ and $x_0\in \Omega$. Then the following
 statements are equivalent:
 	\begin{list}{}{}
 		\item[(i)]  For some $ t_0>0$,
 $\mathcal W_N(\cdot)y_0\not\in L^2_{loc}(t_0,x_0)$;
 		\item[(ii)] For each $t>0$, $\mathcal W_N(\cdot)y_0\not\in L^2_{loc}(t,x_0)$;
 		\item[(iii)] It holds that $A^{-2}y_0\not\in L^2_{loc}(x_0)$.
 	\end{list}
 	\end{proposition}

\begin{proof}
		We first  claim that if $z\in \mathcal H^{-\infty}$ satisfies $z\in L^2_{loc}(x_0)$,
then
	 \begin{align}\label{A-loc-behavior}
	    A^{-1}z \in L^2_{loc}(x_0).
	 \end{align}
	The proof of \eqref{A-loc-behavior} is classical and, for the sake of completeness, we present it as follows:  Write $g:=A^{-1}z$. It is clear that $Ag \in L^2_{loc}(x_0)$. Then
	by  \eqref{selfadjoit-elliptic-operator} and some computations, we can see that $ \Delta g \in L^2_{loc}(x_0)$. Since $\Delta$ is elliptic at $x_0$, one has that $g|_{ B(x_0,r) }\in H^2 \big( B(x_0,r) \big)$ for some $r>0$ (see, for instance, \cite[Theorem 18.1.29]{Hormander-3}). Thus, $g\in L^2_{loc}(x_0)$, which gives that $A^{-1}z \in L^2_{loc}(x_0)$.

	\vskip 5pt
	 We next prove that ($i$)$\Rightarrow$($iii$).
By contradiction, we suppose that  $(i)$ is true, but	
	\begin{align}\label{initial-L2-x0-HN}
	   A^{-2}y_0\in L^2_{loc}(x_0).
	\end{align}
	Since  $M$ is analytic, it follows from (\ref{thm-ODE-meomery-asymptotic-estimate-hypobolic}) that $h_0\equiv 0$ and each $h_l$ is smooth. Then by \eqref{initial-L2-x0-HN} and \eqref{A-loc-behavior}, one has
	\begin{align*}
	\sum_{l=0}^{N-1} h_l(\cdot) (-A)^{-l-1} y_0
 = \sum_{l=1}^{N-1} h_l(\cdot) (-A)^{-l+1} (A^{-2} y_0)
\in L^2_{loc}(t,x_0)\quad\text{for each}\quad
		t> 0.
	\end{align*}
	This, along with the definition of $\mathcal W_N(\cdot)$ (in \eqref{def-PN-HN-RN}),
	implies that
	\begin{align*}
		\mathcal W_N(\cdot)y_0\in L^2_{loc}(t,x_0)
		\quad\text{for each}\quad
		t> 0,
	\end{align*}
	which 	 contradicts ($i$). Therefore, ($iii$) is true.
	
	\vskip 5pt
	We now show that ($iii$)$\Rightarrow$($ii$).
By contradiction, we suppose that  ($iii$) holds, but   $\mathcal W_N(\cdot)y_0\in L^2_{loc}(\hat t_0,x_0)$ for some $\hat t_0> 0$.
	Then from \eqref{def-PN-HN-RN}, we find
	\begin{align}\label{august-ps-weak}
		\sum_{l=0}^{N-1} h_l(\cdot) (-A)^{-l-1}  y_0 \in L^2_{loc}(\hat t_0,x_0).
	\end{align}
	We will use \eqref{august-ps-weak} to prove that
	\begin{align}\label{Sep-A-2-y0-L2-x0}
		A^{-2} y_0 \in L^2_{loc}(x_0).
	\end{align}
	When this is done, we are led to a contradiction with ($iii$), and then the statement  ($ii$) is true.

To  prove \eqref{Sep-A-2-y0-L2-x0}, we observe that there are only two  	 possibilities: either $N=2$ or $N\geq 3$. When
that $N=2$, we see from \eqref{august-ps-weak} and (\ref{thm-ODE-meomery-asymptotic-estimate-hypobolic}) that
	\begin{align*}
		-M(\cdot)A^{-2}y_0
		=\sum_{l=0}^{N-1} h_l(\cdot) (-A)^{-l-1}  y_0
		\in L^2_{loc}(\hat t_0,x_0).
	\end{align*}
	Since $M$ is nonzero, \eqref{Sep-A-2-y0-L2-x0} (for this case) follows by  integrating the above in the time variable.
	
	When $N\geq 3$, it follows from (\ref{thm-ODE-meomery-asymptotic-estimate-hypobolic})
	that $h_0=0$ and $h_1=-M$. Thus, the leading term of the sum in \eqref{august-ps-weak} is  $-M(\cdot) A^{-2}$.
	From  this and the fact that $M$ is nonzero,
	by integrating \eqref{august-ps-weak} in the time variable, we see
	\begin{align*}
	A^{-2} y_0 + \sum_{l=2}^{N-1} c_l A^{-l-1} y_0 \in L^2_{loc}(x_0)
	~~\text{for some}~~
	\{c_l\}_{l=2}^{N-1} \subset \mathbb R.
	\end{align*}
	Write $\mathcal A:= -\sum_{l=2}^{N-1} c_l A^{-l+1}$. 	
	Then the above leads to
	\begin{align}\label{Sep-def-f0-L2}
	f_0:= (1-\mathcal A) \big( A^{-2} y_0 \big) \in L^2_{loc}(x_0).
	\end{align}
	Since $y_0\in \mathcal H^{-\infty}$, there is $m\in \mathbb N^+$ so that
	$y_0 \in \mathcal H^{-m}$.
		This, along with  \eqref{Sep-def-f0-L2}, yields
	\begin{align}\label{Ay0-bary0-relation}
	A^{-2} y_0 - (1+\mathcal A+\cdots + \mathcal A^{m-1}) f_0
	= \mathcal A^{m} (A^{-2} y_0)\in L^2(\Omega).
	\end{align}
	At the same time, since $f_0 \in  L^2_{loc}(x_0)$ in \eqref{Sep-def-f0-L2}, it follows from \eqref{A-loc-behavior} that
	\begin{align*}
		\mathcal A^{j} f_0 \in  L^2_{loc}(x_0)
		~~\text{for each}~~
		j\in \mathbb N.
	\end{align*}
	 Thus, \eqref{Sep-A-2-y0-L2-x0} (for $N\geq 3$) follows from \eqref{Ay0-bary0-relation}.
Hence, \eqref{Sep-A-2-y0-L2-x0} is true.

	\vskip 5pt
	Finally, it is clear that ($ii$)$\Rightarrow$($i$).
Hence, we finish the proof of Proposition \ref{pro-HN-wavelike}.	
\end{proof}

The following Proposition \ref{pro-RN-smooth-2N+2} (which is partially  the conclusion $(iii)$ in Theorem \ref{210320-yb-thm-main-explanations})
 gives an important time-uniform smoothing effect  of the last component  $\mathfrak R_N(\cdot)$.
\begin{proposition}\label{pro-RN-smooth-2N+2}
Let $N\geq 2$ be an integer. Then the last component $\mathfrak R_N$ in \eqref{0423-demcomposition-eq}
exhibits a time-uniform  smoothing effect with a gain of $2N+2$ space derivatives:
	 for each  $y_0 \in \mathcal H^s$ with $s\in \mathbb R$,
	\begin{align}\label{Sep-RN-Aj-etA-smoothing}
		\mathfrak R_N(\cdot)y_0 \in C\big( \overline{ \mathbb R^+ };\mathcal H^{s+2N+2} \big),
		~~\text{while}~~
		A^{-j}e^{\cdot A}y_0,A^{-j}y_0 \in C\big( \overline{ \mathbb R^+ };\mathcal H^{s+2j} \big)
		~(0\leq j \leq N).
	\end{align}
	\end{proposition}

\begin{remark}\label{remark-heat-smoothing-effect}
There is a delicate point worth to clarifying.
Although the map $t\mapsto A^{-j} e^{tA}$, $t\geq 0$ (for a fixed $j$) has infinite order smoothing effect at positive time, it   has the finite time-uniform smoothing effect (given in \eqref{Sep-RN-Aj-etA-smoothing})
 with only $2j$ space-derivatives gained. This index $2j$ is optimal due to the following fact: for  given $y_0 \in \mathcal H^{-\infty}$ and $s\in \mathbb R$, $A^{-j}e^{\cdot A}y_0 \in C\big( \overline{ \mathbb R^+ };\mathcal H^{s+2j} \big)$ if and only if $y_0 \in \mathcal H^s$, which can be directly checked.
\end{remark}

\begin{proof}[Proof of Proposition \ref{pro-RN-smooth-2N+2}]
		The second statement in \eqref{Sep-RN-Aj-etA-smoothing} is clearly true.
We now show  the first statement in \eqref{Sep-RN-Aj-etA-smoothing}.
To this end, we arbitrarily fix  $s\in \mathbb R$, $y_0 \in \mathcal H^s$
and $t_0\geq 0$. We aim to show the continuity of $\mathfrak R_N(\cdot)y_0$ at time $t_0$. 	 Fix any  $\varepsilon>0$. Write $y_0=\sum_{j=1}^{\infty} y_{0,j} e_j$. Since $y_0\in \mathcal H^s$, we can choose  $j_{\varepsilon} \in \mathbb N^+$
large enough so that
	$\sum_{j\geq j_{\varepsilon} } \eta_j^s y_{0,j}^2  < \varepsilon^2$.
		This, along with the definitions of $\mathfrak R_N(\cdot)$ (see \eqref{def-PN-HN-RN}) and
$R_N$ (see \eqref{0921-RN-good-remainder}),
 yields that for each $t\in [0,t_0+1]$,
	\begin{align*}
		&\| \mathfrak R_N(t)y_0 - \mathfrak R_N(t_0)y_0\|_{\mathcal H^{s+2N+2} }^2
		\nonumber\\
		=& \| R_N(t,-A)y_0 - R_N(t_0,-A)y_0\|_{\mathcal H^{s} }^2
\nonumber\\
		=& \sum_{j\geq 1}  \big| R_N(t,\eta_j) - R_N(t_0,\eta_j) \big|^2 y_{0,j}^2 \eta_j^s
		\nonumber\\
		\leq& \max_{1\leq j \leq j_{\varepsilon} }  \big| R_N(t,\eta_j) - R_N(t_0,\eta_j) \big|^2
		\Big( \sum_{j\geq 1} y_{0,j}^2 \eta_j^s \Big)
		+ 4\| R_N \|_{C([0,t_0+1] \times \mathbb R^+ )}^2
		\varepsilon^2.
	\end{align*}
	 From this and \eqref{estimate-p-q-R}, one can find some $C>0$ (independent of $\varepsilon$) and $\delta\in (0,1)$  so that
	\begin{align*}
	\| \mathfrak R_N(t)y_0 - \mathfrak R_N(t_0)y_0\|_{\mathcal H^{s+2N+2} }^2
	\leq C
	\varepsilon^2,~~\mbox{when}\;\; t\in (t_0-\delta,t_0+\delta)\cap \mathbb R^+.
	\end{align*}
	This leads to the continuity of $\mathfrak R_N(\cdot)y_0$ at time $t_0$ and ends the proof
of Proposition \ref{pro-RN-smooth-2N+2}.	
\end{proof}

 \subsection{Proof of main theorems}
  We now give the proofs of Theorems \ref{cor-0423-demcomposition}-\ref{210320-yb-thm-flow-L2H4} one by one.

\begin{proof}[Proof of Theorem \ref{cor-0423-demcomposition}]
First of all, \eqref{2021-april-hl-pl-nontrivial} (with $t>0$) follows from the conclusion $(ii)$ of Theorem \ref{thm-ODE-meomery-asymptotic} at once.

The remainder is to show \eqref{0423-demcomposition-eq}.  	
Fix $s\in \mathbb R$
	and $y_0 \in \mathcal H^s$. Recall that $\eta_j$ is the $j^{th}$ eigenvalue
 of $-A$ and $e_j$ is the corresponding normalized eigenfunction in $L^2(\Omega)$. Then we can write
		\begin{align}\label{4.410-18}
	y_0= \sum_{j\geq 1}  y_{0,j} e_j
	\;\;\mbox{and}\;\;
	y(t;y_0)= \sum_{j\geq 1}  y_j(t) e_j,~~t\geq 0,
	\end{align}
 where $ y_{0,j}:=\langle y_0, e_j \rangle_{\mathcal H^{s},\mathcal H^{-s} }$ and
	$y_j(\cdot)$ satisfies
	\begin{align}\label{4.4,7.24}
	y_j^\prime(t) + \eta_j y_j(t)  +  \int_0^t M(t-\tau) y_j(\tau) d\tau =0,~~t>0;
	~~~y_j(0)=y_{0,j}.
	\end{align}
	By 	(\ref{4.4,7.24}) and (\ref{ode-memory}), we see that for each $j\in\mathbb N^+$,
	\begin{align}\label{ode-pde-memory-relation}
	y_j(t) = y_{0,j}  w_{\eta_j} (t),~~ t\geq 0,
	\end{align}
	where $w_{\eta_j}(\cdot)$ is the solution to (\ref{ode-memory}) with $\eta=\eta_j$.
 	Then by \eqref{4.410-18} and \eqref{ode-pde-memory-relation}, we have
	\begin{align}\label{4.710-18}
	y(t;y_0)=\sum_{j\geq 1}   w_{\eta_j} (t) y_{0,j}  e_j,~~ t\geq 0,
	\end{align}
	From \eqref{4.710-18} and  \eqref{thm-ODE-meomery-asymptotic-estimate} (in Theorem \ref{thm-ODE-meomery-asymptotic}), we see that
when $t\geq 0$,
	\begin{align}\label{5.4,7.25}
	y(t;y_0)=\sum_{j\geq 1}
	\bigg[ e^{-\eta_j t} \bigg(1 + \sum_{l=0}^{N-1} p_{l}(t)  \eta_j^{-l-1}  \bigg)
	+   \sum_{l=0}^{N-1} h_{l}(t)  \eta_j^{-l-1}
	+   R_N(t,\eta_j) \eta_j^{-N-1}
	\bigg]
	y_{0,j}  e_j.
	\end{align}
	Meanwhile, by  functional calculus, we have 	
	\begin{align}\label{5.5,7.25}
	\mbox{RHS of}\; \eqref{5.4,7.25}=\Big[ \Big(e^{tA}   + e^{tA} \sum_{l=0}^{N-1} p_{l}(t)  (-A)^{-l-1}  \Big)
	+   \sum_{l=0}^{N-1} h_{l}(t)  (-A)^{-l-1}
	+   R_N(t,-A) (-A)^{-N-1}
	\Big]y_0.
	\end{align}
	Here, RHS of \eqref{5.4,7.25} denotes the expression on the right-hand side of \eqref{5.4,7.25}.
	
	Since $y_0$ was arbitrarily taken from $\mathcal H^s$,
	we can use  \eqref{5.4,7.25}, \eqref{5.5,7.25}, \eqref{def-PN-HN-RN} and Proposition \ref{prop5.1}
(in Appendix)
	to get   \eqref{0423-demcomposition-eq}.
	This concludes the proof of Theorem \ref{cor-0423-demcomposition}.	
\end{proof}

\begin{proof}[Proof of Theorem \ref{210320-yb-thm-main-explanations}]
 First of all, the conclusions $(i)$-$(ii)$ in Theorem \ref{210320-yb-thm-main-explanations}
follow from  Proposition \ref{pro-PN-heatlike} and Proposition \ref{pro-HN-wavelike},  respectively. In what follows, the conclusions $(iii)$-$(iv)$ in Theorem \ref{210320-yb-thm-main-explanations}
will be proved one by one.

\vskip 5pt
\noindent\textit{
Step 1. The proof of the conclusion $(iii)$ in Theorem \ref{210320-yb-thm-main-explanations}
}

By Proposition \ref{pro-RN-smooth-2N+2}, we only need to show
\eqref{RN-property-regularity}.
For this purpose, we first claim
\begin{align}\label{yb-202104-RN-regularity}
R_N(\cdot,-A)|_{\mathbb R^+}
\in C(\mathbb R^+;\mathcal L( \mathcal H^s))
~\text{for each}~
s\in \mathbb R.
\end{align}
To prove \eqref{yb-202104-RN-regularity}, we arbitrarily fix  $s\in \mathbb{R}$ and $t_2\geq t_1>0$.
Set
\begin{align}\label{5.6,7.25}
\widetilde{\mathcal R}(-A) := R_N(t_2,-A) - R_N(t_1,-A).
\end{align}
Then $\widetilde{\mathcal R}(-A)$ is the operator obtained by the functional calculus of the function:
$R_N(t_2,\cdot) - R_N(t_1,\cdot)$. By \eqref{def-space-with-boundary-condition} and by the spectral representation of $\widetilde{\mathcal R}(-A)$, we see
\begin{align*}
\| \widetilde{\mathcal R} (-A)  \|_{ \mathcal L(\mathcal H^s) }^2
= \sup_{ \sum \eta_j^{s} z_j^2 \leq 1}
\Big\|
\sum_{j\geq 1} \widetilde{\mathcal R} (\eta_j) z_j e_j \Big\|_{ \mathcal H^s }^2
= \sup_{ \sum \eta_j^{s} z_j^2 \leq 1}
\sum_{j=1}^{+\infty}  \eta_j^{s} 		
\big| \widetilde{\mathcal R} (\eta_j) \big|^2 z_j^2
\leq \sup_{\eta >0}  \big| \widetilde{\mathcal R} (\eta) \big|^2.
\end{align*}
This, along with \eqref{5.6,7.25}, yields
\begin{align}\label{5.7,7.25}
\|R_N(t_2,-A) - R_N(t_1,-A) \|_{ \mathcal L(\mathcal H^s) }
= \| \widetilde{\mathcal R} (-A)  \|_{ \mathcal L(\mathcal H^s) }
\leq  \| R_N(t_2,\cdot) - R_N(t_1,\cdot) \|_{C(\mathbb R^+)}.
\end{align}
Since
$R_N \in C(\mathbb R^+; C(\mathbb R^+) )$
(see
Theorem \ref{thm-ODE-meomery-asymptotic}),
\eqref{yb-202104-RN-regularity} follows from \eqref{5.7,7.25} at once.

We now show  \eqref{RN-property-regularity}. Indeed, arguing as in the proof of \eqref{5.7,7.25}, we can obtain
$$
\|R_N(t,-A)\|_{ \mathcal L(\mathcal H^s) }\leq
\| R_N(t,\cdot)\|_{C(\mathbb R^+)},\;\;t\geq 0.
$$
This, along with \eqref{estimate-p-q-R} (in Theorem \ref{thm-ODE-meomery-asymptotic}), leads to \eqref{RN-property-regularity}.

Hence,   the conclusion $(iii)$ in Theorem \ref{210320-yb-thm-main-explanations}
is true.

\vskip 5pt
\noindent\textit{
	Step 2. The proof of the conclusion $(iv)$ in Theorem \ref{210320-yb-thm-main-explanations}
}

Arbitrarily fix $y_0\in  \mathcal H^{-\infty}$,  $x_0\in \Omega$ and $t>0$.
There is an integer $m\geq 2$ so that
\begin{align*}
y_0 \in \mathcal H^{-2m}.
\end{align*}
Then 	by  (\ref{0423-demcomposition-eq}) and \eqref{RN-property-regularity} (where $N$ is replaced by $m$),
one can easily check
\begin{eqnarray}\label{tgy-june-4clock-1}
\varPhi(\cdot)y_0 - \mathcal P_m(\cdot)y_0
- \mathcal W_m(\cdot) y_0
\in L^\infty_{loc}( \overline{ \mathbb R^+ } ; L^2(\Omega) ).
\end{eqnarray}
Combine \eqref{tgy-june-4clock-1} and Proposition \ref{pro-PN-heatlike} (where  $N$ is replaced by $m$) to get
\begin{align}\label{Sep-varPhi-Hm-L2}
\varPhi(\cdot)y_0 \not\in L^2_{loc}(t,x_0)
\Leftrightarrow
\mathcal W_m(\cdot)y_0 \not\in L^2_{loc}(t,x_0).
\end{align}
Meanwhile,
for each integer $k\geq 2$,
it follows by Proposition \ref{pro-HN-wavelike} (where $N$ is replaced by $k$) that
\begin{align*}
\mathcal W_k(\cdot)y_0 \not\in L^2_{loc}(t,x_0)
\Leftrightarrow
A^{-2}y_0 \not\in L^2_{loc}(x_0),
\end{align*}
in particular,
\begin{align*}
\mathcal W_m(\cdot)y_0 \not\in L^2_{loc}(t,x_0)
\Leftrightarrow
\mathcal W_N(\cdot)y_0 \not\in L^2_{loc}(t,x_0).
\end{align*}
This, along with \eqref{Sep-varPhi-Hm-L2}, yields
	\begin{align}\label{varPhi-HN-equivalence}
\varPhi(\cdot)y_0\not\in L^2_{loc}(t,x_0)  \Leftrightarrow
\mathcal W_N(\cdot)y_0\not\in L^2_{loc}(t,x_0).
\end{align}
Thus, the conclusion $(iv)$ follows from \eqref{varPhi-HN-equivalence} and Proposition \ref{pro-HN-wavelike}
at once.

\vskip 5pt
Hence, we complete the proof of Theorem \ref{210320-yb-thm-main-explanations}.
\end{proof}

%
%
%
%

 \begin{proof}[Proof of  Theorem \ref{210320-yb-thm-main-high-frequency-behaviors}]
Arbitrarily fix an integer $N\geq 2$.
We first show \eqref{yb-march-initial-behaviors-limit}. Indeed, one can see from  \eqref{varPhi-y-y0},  \eqref{def-PN-HN-RN} and \eqref{0921-RN-good-remainder} that for each $j\in \mathbb N^+$,
\begin{align*}
\begin{array}{l}
\varPhi(0)e_j=e_j,~
\mathcal P_N(0)e_j=\Big( 1+ \sum_{l=0}^{N-1} p_l(0) \eta_j^{-l-1} \Big) e_j,
\\
\mathfrak R_N(0)e_j=0,~
\mathcal W_N(0)e_j=\Big( \sum_{l=0}^{N-1} h_l(0) \eta_j^{-l-1} \Big) e_j.
\end{array}
\end{align*}
Since $\displaystyle\lim_{j\rightarrow+\infty} \eta_j=+\infty$,
the above leads to
\eqref{yb-march-initial-behaviors-limit} at once. 


We next prove \eqref{210320-yb-positive-behaviors-limit}.
Arbitrarily fix $t>0$.
On one hand,
it follows
from  \eqref{thm-ODE-meomery-asymptotic-estimate-hypobolic} that $h_0\equiv 0$ and $h_1(t)=-M(t)$, which, along with \eqref{def-PN-HN-RN}, yield
\begin{align*}
  \mathcal W_N(t)e_{j}= -M(t) \eta_j^{-2} e_{j}
  + \sum_{1<l\leq N-1} h_l(t) \eta_j^{-l-1} e_{j}\;\;\mbox{for all}\;\;j\geq 1.
\end{align*}
This implies
\begin{align}\label{210322-thm2-2ndEq-2.1}
   \lim_{j\rightarrow+\infty} \|\mathcal W_N(t)e_{j}\|_{\mathcal H^4}
   = |M(t)|
   ~~\text{and}~~
   \lim_{j\rightarrow+\infty} \|\mathcal W_N(t)e_{j}\|_{\mathcal H^s}=0,~\forall\, s<4.
\end{align}
On the other hand, from \eqref{def-PN-HN-RN}, it follows that for each $j\geq 1$,
\begin{align*}
\mathcal P_N(t)e_{j}=  e^{-t\eta_j} \Big(1+ \sum_{l=0}^{N-1} p_l(t) \eta_j^{-l-1} \Big)  e_j  ~~\text{and}~~
\mathfrak R_N(t)e_j= \eta_j^{-N-1} R_N(t,-A)  e_j.
\end{align*}
Since $\displaystyle\lim_{j\rightarrow+\infty} \eta_j^{-1}=
\displaystyle\lim_{j\rightarrow+\infty} \eta_j^{\frac{s}{2}} e^{-t\eta_j}
=0$ ($s\in \mathbb R$),
 the above, together with \eqref{RN-property-regularity}, gives
\begin{align}\label{210322-thm2-2ndEq-2.2}
 \displaystyle\lim_{j\rightarrow+\infty} \|\mathcal P_N(t)e_{j}\|_{\mathcal H^{s_1}}
 = \displaystyle\lim_{j\rightarrow+\infty} \|\mathfrak R_N(t)e_{j}\|_{\mathcal H^{s_2}}
=0,~s_1\in \mathbb R,~s_2<2N+2.
\end{align}

 Now, by  \eqref{210322-thm2-2ndEq-2.1} and \eqref{210322-thm2-2ndEq-2.2}, we can use the decomposition
\eqref{0423-demcomposition-eq} (with the above $N\geq 2$) to get
\begin{align*}
\lim_{j\rightarrow+\infty} \|\varPhi(t)e_{j}\|_{\mathcal H^4}
= |M(t)|
~~\text{and}~~
\lim_{j\rightarrow+\infty} \|\varPhi(t)e_{j}\|_{\mathcal H^s}=0,~\forall\, s<4.
\end{align*}
These, along with \eqref{210322-thm2-2ndEq-2.1}-\eqref{210322-thm2-2ndEq-2.2}, lead to \eqref{210320-yb-positive-behaviors-limit}.
 This
completes the proof of Theorem \ref{210320-yb-thm-main-high-frequency-behaviors}.
\end{proof}

 \begin{proof}[Proof of Theorem \ref{210320-yb-thm-flow-L2H4}]
  Arbitrarily fix $s\in \mathbb{R}$. We will prove the conclusions $(i)$-$(iii)$  one by one.

\vskip 5pt
(i)  Arbitrarily fix  $\alpha\in[0,4]$ and $t>0$.
We apply \eqref{0423-demcomposition-eq} (with $N=2$), as well as \eqref{def-PN-HN-RN} and \eqref{thm-ODE-meomery-asymptotic-estimate-hypobolic}, to obtain
\begin{align}\label{march-phi-N=2-1}
\varPhi(t)&= \Big( e^{tA} -p_0(t) A^{-1} e^{tA} + p_1(t) A^{-2} e^{tA} \Big) -
M(t)A^{-2}
+ R_2(t,-A) (-A)^{-3}.
\end{align}
Meanwhile, notice that
\begin{align*}
\| A^{\frac{\alpha}{2} -j } e^{tA} \|_{\mathcal L( \mathcal H^s )}
\Big(=\sup_{j\in \mathbb N^+} \eta_j^{\frac{\alpha}{2} -j } e^{-t\eta_j} \Big)
\leq t^{j-\frac{\alpha}{2} }
   ~\text{and}~
   \| A^{\frac{\alpha}{2} -j -2} \|_{\mathcal L( \mathcal H^s )} \leq \eta_1^{\frac{\alpha}{2} -j -2}, ~j=0,1,2.
\end{align*}
These, along with \eqref{march-phi-N=2-1}, yield
\begin{align*}
   \| \varPhi(t)  \|_{\mathcal L(\mathcal H^s,\mathcal H^{s+\alpha})}
   \leq  t^{-\frac{\alpha}{2} }\Big( 1 +  t|p_0(t)| + t^2|p_1(t)|  \Big)
   + \eta_1^{\frac{\alpha}{2} -2 } |M(t)|+ \eta_1^{\frac{\alpha}{2} -3 }  \|R_2(t,\cdot)\|_{ \mathcal L( \mathcal H^{s+\alpha} )}.
\end{align*}
After direct computations, we obtain \eqref{march-flow-H4-estimate} from the above, \eqref{thm-ODE-meomery-asymptotic-estimate-hypobolic} and \eqref{RN-property-regularity} (where $N=2$).

\vskip 5pt
(ii) Assume that \eqref{march-flow-H4-estimate-1} holds for $t\in (0,\delta_0)$.
We first claim
\begin{align}\label{march-flow-smoothing-order}
 +\infty>\limsup_{j\rightarrow+\infty} \eta_j^{\frac{\alpha_0}{2}-2},
 ~~\text{i.e.,}~~ \alpha_0\leq 4.
\end{align}
Indeed, by the assumption $(\mathfrak C)$, we can choose small $t_0\in(0,\delta_0)$ so that
\begin{align}\label{yb-202104-M(t0)-nonzero}
M(t_0)\neq 0.
\end{align}
Then from \eqref{march-phi-N=2-1} and \eqref{RN-property-regularity} (where $N=2$), it follows that for each $j\geq 1$,
\begin{align*}
  \| \varPhi(t_0) \|_{\mathcal L(\mathcal H^s,\mathcal H^{s+\alpha_0}) }
  \geq\| \varPhi(t_0) e_j\|_{ \mathcal H^{s+\alpha_0}}   / \|e_j\|_{\mathcal H^s}
  = \eta_j^{\frac{\alpha_0}{2}-2} \Big(
  |M(t_0)| + O(\eta_j^{-1})  \Big),
\end{align*}
(Here and what follows,  by $O(\eta_j^{-1})$ with $j\geq 1$, we mean that there is $C_1>0$, independent of $j$, so that $| O(\eta_j^{-1}) | \leq C_1 \eta_j^{-1} $.) This, along with  \eqref{march-flow-H4-estimate-1} (where $t=t_0$), \eqref{yb-202104-M(t0)-nonzero}
and the fact that $\displaystyle\lim_{j\rightarrow+\infty} \eta_j=+\infty$, leads to
\eqref{march-flow-smoothing-order}.

We next claim
 \begin{align}\label{march-flow-blowup-order-2}
   \displaystyle\liminf_{t\rightarrow 0^+}  t^{\frac{\alpha_0}{2}} \| \varPhi(t) \|_{\mathcal L(\mathcal H^s,\mathcal H^{s+\alpha_0})} >0.
 \end{align}
 In fact, from \eqref{march-flow-H4-estimate-1} (where $z=e_j$), \eqref{march-phi-N=2-1} and \eqref{RN-property-regularity}, we
 see that when $t\in(0,\delta_0)$ and $j\in \mathbb N^+$,
 \begin{align}\label{march-flow-blowup-order-1}
      t^{\frac{\alpha_0}{2}} \| \varPhi(t) \|_{\mathcal L(\mathcal H^s,\mathcal H^{s+\alpha_0})}
     \geq&
      t^{\frac{\alpha_0}{2}} \| \varPhi(t) e_j\|_{\mathcal H^{s+\alpha_0}}
      / \|e_j\|_{\mathcal H^s}
      \nonumber\\
    \geq& (t\eta_j)^{\frac{\alpha_0}{2}}  \Big( \big| 1+ O(\eta_j^{-1}) \big|
      e^{-t\eta_j }
      -   O(\eta_j^{-1})   \Big).
 \end{align}

 Meanwhile, it follows by Weyl's asymptotic formula for the eigenvalues (see for instance \cite[XIII.15]{SIMON2})
 that  $\displaystyle\lim_{j\rightarrow+\infty} \eta_{j+1}/\eta_j=1$. Thus,
 there is $j_0\in \mathbb N^+$ so that
 \begin{align*}
    [\eta_{j_0},+\infty) = \cup_{j\geq j_0} [\eta_j, 2\eta_j).
 \end{align*}
 Therefore, for each  $t\in (0,\eta_{j_0}^{-1}]$, there is an integer $j_t\geq j_0$ so that
 \begin{align*}
    t^{-1}\in [\eta_{j_t}, 2\eta_{j_t}),
    ~\text{i.e.,}~
    \frac{1}{2} < t \eta_{j_t} \leq 1.
 \end{align*}
 This, along with \eqref{march-flow-blowup-order-1}, leads to  \eqref{march-flow-blowup-order-2}.

 Finally, the conclusion $(ii)$ follows from  \eqref{march-flow-smoothing-order} and \eqref{march-flow-blowup-order-2}
 at once.

\vskip 5pt
(iii) We first claim
\begin{align}\label{yb-202104-solution-continutity-in-H4}
  \varPhi(\cdot) y_0 \in C(\mathbb R^+; \mathcal H^{s+4})
  ~\text{for each}~
  y_0\in \mathcal H^s.
\end{align}
To this end, we arbitrarily fix  $y_0\in \mathcal H^s$, $t_0>0$ and $\varepsilon>0$. Since $y_0\in \mathcal H^s$, there is
$j_{\varepsilon} \in \mathbb N^+$ so that
\begin{align}\label{yb-202104-y0-y0j-error}
 \| y_0 - y_{0,j_{\varepsilon}} \|_{\mathcal H^s} < \varepsilon
 ~\text{where}~
 y_{0,j_{\varepsilon}} := \sum_{j=1}^{j_{\varepsilon}}
 \langle y_0,e_j \rangle_{\mathcal H^s, \mathcal H^{-s}} e_j.
\end{align}
At the same time,  one can easily check $\varPhi(\cdot) y_{0,j_{\varepsilon}} \in C(\mathbb R^+; \mathcal H^{s+4})$, which implies there is  $\delta_{\varepsilon} \in (0,t_0/2)$ so that
\begin{align*}
\sup_{t_0-\delta_{\varepsilon}<t< t_0+\delta_{\varepsilon}}
 \| \varPhi(t) y_{0,j_{\varepsilon}}  -  \varPhi(t_0) y_{0,j_{\varepsilon}}  \|_{\mathcal H^{s+4}}  < \varepsilon.
\end{align*}
This, together with  \eqref{yb-202104-y0-y0j-error}, gives
\begin{align*}
\sup_{t_0-\delta_{\varepsilon}<t< t_0+\delta_{\varepsilon}}
\| \varPhi(t) y_{0}  -  \varPhi(t_0) y_{0}  \|_{\mathcal H^{s+4}}  <
\varepsilon   +   2 \sup_{t_0/2\leq \tau\leq 2t_0}  \| \varPhi(\tau)\|_{\mathcal L(\mathcal H^s, \mathcal H^{s+4})} \varepsilon,
\end{align*}
which, along with \eqref{march-flow-H4-estimate}, leads to  \eqref{yb-202104-solution-continutity-in-H4}.

We next claim that for each  $\alpha>4$,
\begin{align}\label{yb-202104-tj-solution-not-regular}
 \varPhi(\cdot) \hat y_0 \not\in C(\mathbb R^+; \mathcal H^{s+\alpha})
 ~\text{for some}~
 \hat y_0\in \mathcal H^s.
\end{align}
For this purpose, we arbitrarily fix  $\alpha>4$.
Since $h_0(t)\equiv 0$ and $h_1(t)=-M(t)$ (see \eqref{thm-ODE-meomery-asymptotic-estimate-hypobolic}),
we apply \eqref{0423-demcomposition-eq} (with $N=2$), as well as \eqref{def-PN-HN-RN} and \eqref{RN-property-regularity}, to obtain
\begin{align}\label{4.30,4-18}
   \lim_{j\rightarrow +\infty}  \eta_j^2\langle \varPhi(t) e_j, e_j  \rangle_{L^2(\Omega)}
   = -M(t)
   ~\text{for each}~
   t>0.
\end{align}
Meanwhile, by the assumption ($\mathfrak C$), we can choose  $\hat t_0>0$ so that $M(\hat t_0)\neq 0$.
This, along with \eqref{4.30,4-18} and the fact that  $\alpha>4$,  yields
\begin{align*}
   \eta_j^{\frac{\alpha}{2}} |\langle \varPhi(\hat t_0) e_j, e_j  \rangle_{L^2(\Omega)}|
  \rightarrow +\infty
  ~\text{as}~
  j\rightarrow+\infty.
\end{align*}
Thus we can choose a subsequence  $\{k_j\}_{j\geq 1}$ of $\mathbb N^+$ so that
\begin{align}\label{yb-202104-tj-solution-not-regular-pf-1}
   \eta_{k_j}^{\frac{\alpha}{2}} |\langle \varPhi(\hat t_0) e_{k_j}, e_{k_j}  \rangle_{L^2(\Omega)}|
   \geq j2^j  ~\text{for each}~  j\geq 1.
\end{align}

We now define $\hat y_0\in \mathcal H^s$ as follows:
\begin{align}\label{yb-202104-tj-solution-not-regular-pf-2}
 \hat y_0  := \sum_{j\geq 1}  2^{-j} \eta_{k_j}^{-\frac{s}{2}}  e_{k_j}.
\end{align}
Then, by  \eqref{yb-202104-tj-solution-not-regular-pf-1} and \eqref{yb-202104-tj-solution-not-regular-pf-2},
we find
\begin{align*}
 \| \varPhi(\hat t_0) \hat y_0 \|_{\mathcal H^{s+\alpha}}^2
 =&  \sum_{j\geq 1}  \eta_{k_j}^{ s+\alpha }
\Big( \langle \varPhi(\hat t_0) e_{k_j}, e_{k_j} \rangle_{L^2(\Omega)}
\langle \hat y_0,e_{k_j} \rangle_{\mathcal H^s, \mathcal H^{-s}} \Big)^2
\nonumber\\
=& \sum_{j\geq 1}
\Big( \eta_{k_j}^{\frac{\alpha}{2}} \langle \varPhi(\hat t_0) e_{k_j}, e_{k_j}  \rangle_{L^2(\Omega)}  2^{-j}
 \Big)^2
 \geq \sum_{j\geq 1} j^2
 = +\infty ,
\end{align*}
which implies $\varPhi(\hat t_0) \hat y_0 \not\in \mathcal H^{s+\alpha}$.
This leads to \eqref{yb-202104-tj-solution-not-regular}.

\vskip 5pt

 In conclusion, we complete the proof of  Theorem \ref{210320-yb-thm-flow-L2H4}.
\end{proof}

\begin{remark}\label{remark-yb-202104-special-points}	
By a very similar way to that used in the proof of  \eqref{march-flow-H4-estimate} (in Theorem \ref{210320-yb-thm-flow-L2H4}), we can also show what follows:
For each $t>0$,
\begin{align}\label{4.33,4-18}
	 \varPhi(t) \in \mathcal L(\mathcal H^s, \mathcal H^{s+2k(t)+2})
	 ~\text{for each}~ s\in \mathbb R,
	\end{align}
where $k(t):=\min\{ l \geq 1 ~:~   h_l(t) \neq 0\}$.

From \eqref{4.33,4-18} and the fact that $h_1(t)=-M(t)$ (which follows from \eqref{thm-ODE-meomery-asymptotic-estimate-hypobolic}),
we conclude that
{\it the smoothing effect of the flow $\varPhi(t)$ at points  in the set $\{t>0~:~M(t)=0\}$
    is better than that at points in the set
        $\{t>0~:~M(t)\neq 0\}$.}

\end{remark}

\subsection{Other properties of the flow and the components}

This subsection presents more properties of the flow and the components.
More precisely, first, we formulate Propositions \ref{pro-varPhi-dominatings-by-PN}
 to   illustrate
 how the component $\mathcal P_N$ influences the flow,
from the perspective of the singularities; (In plain language, it tells us that the singularity of
$\mathcal P_N(0)$ determines the singularity of $\varPhi(0)$.);
second, we give Proposition \ref{prop-PW-nonprojections} to show that both $\mathcal P_N(0)$
and $\mathcal W_N(0)$ are not projections;
last, we present Propositions \ref{prop-varPhi-expression}-\ref{cor-evolution-high-regularity},
which  might have independent interest.

\begin{proposition}\label{pro-varPhi-dominatings-by-PN}
	Let $N\geq 2$ be an integer and let $y_0\in  \mathcal H^{-\infty}$ and   $x_0\in \Omega$. Then
	  $(y_0=)\,\varPhi(0)y_0\not\in L^2_{loc}(x_0)$ if and only if $\mathcal P_N(0)y_0\not\in L^2_{loc}(x_0)$.
\end{proposition}

\begin{proof}
	 First of all, since $\varPhi(0)y_0=y_0$, we have
		\begin{align}\label{0921-varPhi(0)-y0}
		\varPhi(0)y_0 \not\in L^2_{loc}(x_0)  \Leftrightarrow
		y_0 \not\in L^2_{loc}(x_0).
		\end{align}
		We claim
		\begin{align}\label{0921-PN(0)-y0}
		\mathcal P_N(0)y_0 \not\in L^2_{loc}(x_0)  \Leftrightarrow
		y_0 \not\in L^2_{loc}(x_0)
		\end{align}
		When this is proved, the conclusion in this proposition  follows from \eqref{0921-varPhi(0)-y0} and \eqref{0921-PN(0)-y0}
at once.
		
		We now  show \eqref{0921-PN(0)-y0}.
To prove the necessity,
we suppose, by contradiction, that  the statement on the left-hand side of \eqref{0921-PN(0)-y0} is true, but $y_0\in L^2_{loc}(x_0)$.  Then we apply \eqref{A-loc-behavior} to obtain that
		\begin{align*}
			A^{-j}y_0 \in L^2_{loc}(x_0)
			~~\text{for each}~~
			j\in \mathbb N.
		\end{align*}
		At the same time, it follows from
		\eqref{def-PN-HN-RN} that
		\begin{align*}
		\mathcal P_N(0)y_0=y_0 + \sum_{l=0}^{N-1} p_l(0) (-A)^{-l-1}y_0.
		\end{align*}
		These imply that  $\mathcal P_N(0)y_0 \in L^2_{loc}(x_0)$, which  contradicts the statement on the left-hand side of \eqref{0921-PN(0)-y0}. Therefore, we have shown the necessity.
		
		To show the sufficiency, we suppose, by contradiction, that  the statement on the right-hand side of \eqref{0921-PN(0)-y0}
holds, but $\mathcal P_N(0)y_0\in L^2_{loc}(x_0)$. Then from \eqref{def-PN-HN-RN}, there is a sequence $(\hat c_l)_{l=0}^{N-1}\subset \mathbb R$ so that
		\begin{align*}
			y_0+\sum_{l=0}^{N-1} \hat c_l A^{-l-1} y_0 \in L^2_{loc}(x_0).
		\end{align*}
		Then  by the similar way as that used in the proof of ``\eqref{august-ps-weak} $\Rightarrow$ \eqref{Sep-A-2-y0-L2-x0}", one can get that $y_0\in L^2_{loc}(x_0)$. This contradicts the statement on the right-hand side of \eqref{0921-PN(0)-y0}. Therefore, the sufficiency is proved. Hence, we finish  the proof of Proposition \ref{pro-varPhi-dominatings-by-PN}.		\end{proof}

\begin{proposition}\label{prop-PW-nonprojections}
The following conclusions are true:

\noindent $(i)$ For each  integer $N\geq 2$,
\begin{align}\label{PWR-initial-time}
\mathfrak R_N(0)=0
~~\text{and}~~
  y_0 =  \mathcal P_N(0)y_0  +  \mathcal W_N(0) y_0\;\;\mbox{for each}\;\; y_0 \in L^2(\Omega).
\end{align}

\noindent $(ii)$ For each  integer $N\geq 2$,
 $\mathcal P_N(0)$ and $\mathcal W_N(0)$ are projections over $L^2(\Omega)$ if and only if
\begin{align}\label{M-derivatives-N-zero}
  M(0)=\cdots=M^{(j)}(0)=\cdots=M^{(N-2)}(0)=0.
\end{align}
\end{proposition}

\begin{proof}
 First, the conclusion $(i)$  follows from  \eqref{def-PN-HN-RN}, \eqref{0921-RN-good-remainder} and \eqref{0423-demcomposition-eq} (with $t=0$) at once.

 We next prove the conclusion $(ii)$.
 To show the sufficiency, we assume that \eqref{M-derivatives-N-zero} is true.
  Then by \eqref{M-derivatives-N-zero}, we can apply \eqref{convolution-two-ineqs} (with $t\rightarrow 0^+$) to get
\begin{align*}
  \frac{d^{l} }{dt^l}  \underset{j}{ \underbrace{M*\cdots*M} }(0)=0,~j\in \mathbb N^+,~l\in\{0,\ldots,N-2\}.
\end{align*}
This, along with \eqref{thm-ODE-meomery-asymptotic-estimate-hypobolic}, yields
\begin{align*}
   h_l(0)=0,~~0\leq l \leq N-1,
\end{align*}
from which and \eqref{def-PN-HN-RN}, it follows $\mathcal W_N(0)=0$ and consequently, $\mathcal W_N(0)$
is a projection. Then by the second equality in \eqref{PWR-initial-time},
we see that $\mathcal P_N(0)$ is a projection.

To prove the necessity, we suppose
 that $\mathcal P_N(0)$ and $\mathcal W_N(0)$ are projections over $L^2(\Omega)$. Then we have
\begin{align*}
   \mathcal W_N(0)^2y_0= \mathcal W_N(0)y_0\;\;\mbox{for each}\;\;y_0\in L^2(\Omega).
\end{align*}
Taking $y_0=e_j$ with $j\in \mathbb N^+$ in the above and using  \eqref{def-PN-HN-RN},  we find
\begin{align*}
  \bigg( \sum_{l=0}^{N-1} h_l(0) \eta_j^{-l-1}  \bigg)^2 e_j &= \bigg( \sum_{l=0}^{N-1} h_l(0) (-A)^{-l-1}  \bigg)^2 e_j
  = \mathcal W_N(0)^2y_0
  \nonumber\\
  &=\mathcal W_N(0)y_0
  = \bigg( \sum_{l=0}^{N-1} h_l(0) (-A)^{-l-1}  \bigg) e_j
   =\bigg( \sum_{l=0}^{N-1} h_l(0) \eta_j^{-l-1}  \bigg) e_j.
\end{align*}
This implies
\begin{align*}
  \bigg( \sum_{l=0}^{N-1} h_l(0) \eta_j^{-l-1}  \bigg)^2
   = \sum_{l=0}^{N-1} h_l(0) \eta_j^{-l-1}\;\;\mbox{for all}\;\;j\geq 1.
\end{align*}
Since $\displaystyle\lim_{j\rightarrow+\infty} \eta_j=+\infty$,  the above, divided by $\eta_j^{-1},\ldots,\eta_j^{-N}$ respectively, gives  $h_l(0)=0$, $0\leq l\leq N-1$. Then by direct computations and by \eqref{thm-ODE-meomery-asymptotic-estimate-hypobolic},  we get \eqref{M-derivatives-N-zero}. This ends the proof of Proposition \ref{prop-PW-nonprojections}.
\end{proof}

\begin{proposition}\label{prop-varPhi-expression}
	Let $K_M$ be given by (\ref{new-kernel-KM}). Then
	\begin{align}\label{eq-varPhi-expression}
	\varPhi(t)^*=\varPhi(t)
	= e^{tA} +   \int_0^t K_M(t,\tau) e^{\tau A} d\tau,
	\;\;  t\geq 0.
	\end{align}
		Moreover, it holds that
	\begin{align}\label{varPhi-Hs-estimate}
	 \sup_{s\in \mathbb R} \| \varPhi(t) - e^{tA} \|_{ \mathcal L( \mathcal H^s ) }
	\leq   \inf_{ \lambda \geq -\eta_1 }
	e^{\lambda t}  \bigg[ \exp\Big( t
	\int_0^t e^{-\lambda \tau} |M(\tau)| d\tau \Big)
	- 1  \bigg],~~
	t\geq 0.
	\end{align}
	(Here, $-\eta_1$ is the first eigenvalue of $A$.)
	\end{proposition}

\begin{proof}
	Arbitrarily fix $s\in \mathbb R$ and $y_0 \in \mathcal H^s$.
By \eqref{4.410-18}, \eqref{4.710-18}
	 and  Lemma \ref{lem-ode-memory-infinite-series},
	we find
	\begin{align}\label{4.6,7.24}
	y(t;y_0)= \sum_{j\geq 1}  w_{\eta_j} (t) y_{0,j}   e_j
	=  \sum_{j\geq 1}  \Big( e^{-\eta_j t} +
	\int_0^t K_M(t,\tau)  e^{-\eta_j \tau} d\tau
	\Big) y_{0,j} e_j,\;\;t\geq 0.
	\end{align}
	Now the second equality in \eqref{eq-varPhi-expression} follows from (\ref{4.6,7.24})
	and (\ref{varPhi-y-y0}), while the first equality in \eqref{eq-varPhi-expression}
	follows by
	the second one in \eqref{eq-varPhi-expression}
	and the fact that
	for each $t\geq 0$, $e^{tA}=e^{tA^*}$ in $\mathcal L(\mathcal H^s)$.

	Finally,  it follows from \eqref{eq-varPhi-expression} that
	\begin{align*}
	\| \varPhi(t) - e^{tA} \|_{ \mathcal L( \mathcal H^s ) }
	\leq \int_0^t \| e^{\tau A} \|_{ \mathcal L( \mathcal H^s ) }
	| K_M(t,\tau) | d\tau
	\leq \int_0^t e^{- \eta_1 \tau}
	| K_M(t,\tau) |  d\tau,\; t\geq 0.
	\end{align*}
	This, along with 	\eqref{KM-regularity-estimate}
in Proposition \ref{prop-KM-regularity}, where $\lambda\geq -\eta_1$, yields
	\begin{align*}
	\| \varPhi(t) - e^{tA} \|_{ \mathcal L( \mathcal H^s ) }
	\leq e^{\lambda t}
	\int_0^t  e^{-\lambda (t-\tau)}
	| K_M(t,\tau) | d\tau
	\leq e^{\lambda t}
	\bigg[
	\exp\bigg( t  \int_0^t e^{-\lambda \tau} |M(\tau)| d\tau
	\bigg) - 1  \bigg], \; t\geq 0,
	\end{align*}
	which leads to \eqref{varPhi-Hs-estimate} and completes the proof of Proposition \ref{prop-varPhi-expression}.	
\end{proof}

\begin{proposition}\label{cor-evolution-high-regularity}
Let $s\in \mathbb{R}$. Then $\varPhi(\cdot)$ is real analytic from $\mathbb R^+$ to $\mathcal L( \mathcal H^s )$.
\end{proposition}

\begin{proof}
	 According to Proposition \ref{prop-KM-analytic}, $K_M$ is real analytic over
	$S_+$. Then by \eqref{eq-varPhi-expression} and
	the analyticity of $\{e^{tA}\}_{t> 0}$, we obtain
	that $\varPhi(\cdot)$ is real analytic from $\mathbb R^+$ to $\mathcal L( \mathcal H^s )$.
	This ends the proof of Proposition \ref{cor-evolution-high-regularity}.	
\end{proof}

\subsection{Less regular memory kernels}

The techniques of this paper can also be employed to handle less regular memory kernels. Assume that:
\begin{itemize}
	\item[] $(\mathfrak C_1)$ ~~The memory kernel $M$ is in $C^{N_0}(\overline{\mathbb R^+})$ for a fixed integer $N_0\geq 2$.
\end{itemize}
Then, the  following holds, and can be proved by the  same arguments of the proofs of  Theorems \ref{cor-0423-demcomposition}-\ref{210320-yb-thm-main-explanations}.

\begin{theorem}\label{thm-decomposition-extension}
	Suppose that  $(\mathfrak C_1)$ is true. Then for each $N\in \{2,\ldots,N_0\}$,
all results of Theorems \ref{cor-0423-demcomposition}-\ref{210320-yb-thm-main-explanations}, except for \eqref{2021-april-hl-pl-nontrivial}-\eqref{yb-202104-wn-propagation} and
\eqref{202104-varphi-wn-y0},   are true.
\end{theorem}

\section{An example}\label{example}

In this section we analyze with more details the decomposition \eqref{0423-demcomposition-eq} in the particular case of
the following  memory kernel:
\begin{align*}
  M(t)=\alpha e^{\lambda t},~~t\geq0,\;\;\mbox{where}\;\;\lambda,\alpha\in\mathbb R,\;\;\mbox{with}\;\alpha\neq0.
\end{align*}
In what follows, we adopt the conventional notation:  $0^0:=1$ and $\displaystyle\sum_{\emptyset}\cdot:=0$.

 In this particular case,
 we have the following explicit expressions:
  For each $l\in \mathbb N$,
\begin{eqnarray}\label{M=exp-hl-pl}
	\left\{
	\begin{array}{l}
		h_l(t)= e^{\lambda t}  (-1)^l
		\displaystyle\sum_{ \tiny\begin{array}{c}
				m,k\in \mathbb N, j\in\mathbb N^+\\
				k-m+2j=l+1,\\
				m+k\leq l-1
			\end{array}
		}
		C_l^{l-j}
		\alpha^j \lambda^{k}
		\frac{t^{m}} {m!},
		~~t\geq0,\\
		p_l(t) =
		(-1)^{l+1}
		\displaystyle\sum_{ \tiny\begin{array}{c}
				m,k\in \mathbb N, j\in\mathbb N^+\\
				k-m+2j=l+1,\\
				m+k\leq l+1
			\end{array}
		}
		C_l^{l-j+m}
		\alpha^j \lambda^{k}
		\frac{(-t)^{m}} {m!},
		~~t\geq0,\\
	R_N(t,\tau)=(-1)^N N! \int_0^t \tau e^{-\tau s}
	e^{\lambda(t-s)}\mathcal{F}_N(t,s)ds,\;\; t,\tau\geq 0,	
	\end{array}
	\right.
\end{eqnarray}
where
\begin{eqnarray*}
	\mathcal{F}_N(t,s):=\sum_{ \tiny\begin{array}{c}
			\beta_1,\beta_2,\beta_3\in \mathbb N, j\in\mathbb N^+\\
			\beta_1+\beta_2+\beta_3=N,\\
			\beta_1\leq j,\beta_2\leq j-1
		\end{array}
	}
	\frac{(-s)^{j-\beta_1}}{(j-\beta_1)!} \frac{(t-s)^{j-1-\beta_2}}{(j-1-\beta_2)!}
	\lambda^{\beta_3} \alpha^j.
\end{eqnarray*}

For this example, the following holds:

\begin{proposition}\label{prop-a.e.t-nontrivial-main-part}
	For almost every $t\geq 0$ and for each $l\in \mathbb N^+$, it holds that
	\begin{align}\label{nontrivial-a.e.t-hl-pl}
		h_l(t)\neq 0 ~~\text{and}~~
		p_l(t)\neq 0.
	\end{align}
	
\end{proposition}

\begin{proof}
Fix $l\geq 1$ arbitrarily. It follows from  \eqref{M=exp-hl-pl} that  the function $t\mapsto e^{-\lambda t} h_l(t)$  ($t\geq 0$) is a polynomial of the order $l-1$
and satisfies
	\begin{align*}
		e^{-\lambda t} h_l(t) = (-\alpha)^l \frac{t^{l-1} }{(l-1)!} + \sum_{m=0}^{l-2} \cdots t^m ,~~t\geq0.
	\end{align*}
	Thus $h_l(t)\neq 0$ for a.e. $t\geq 0$, which implies the first inequality in \eqref{nontrivial-a.e.t-hl-pl}.  The same can be done for $p_l$ by the similar argument. This ends the proof. 	
\end{proof}

\begin{remark}\label{remark5.2}

The following comments on Proposition \ref{prop-a.e.t-nontrivial-main-part}  are worth considering.
\begin{itemize}
	\item[(R1)]
As $N$ increases, for almost every $t\geq 0$, both $\mathcal P_N(t)$ and $\mathcal W_N(t)$ involve more and more nontrivial terms.

	\item[(R2)]
For each $N\in \mathbb{N}^+$ and for a.e. $t\geq 0$, the remainder $\mathfrak R_N(t)$ inherits
	a hybrid heat/wave structure (and thus is not negligible).
Indeed,  from
\eqref{0423-demcomposition-eq}-\eqref{def-PN-HN-RN}, one can directly obtain    the following recursive equality:
	\begin{align}\label{RN-recursive-eq}
		\mathfrak R_j(t)=  \mathfrak R_{j+1}(t)  +   p_{j}(t) e^{tA}(-A)^{-j-1} + h_j(t) (-A)^{-j-1},~t\geq 0,~j\in \mathbb N^+.
	\end{align}
From this and \eqref{nontrivial-a.e.t-hl-pl}, it follows that for each $N\in \mathbb N^+$ and for a.e. $t\geq 0$,  $\mathfrak R_N(t)$ contains both nontrivial terms in  $\mathcal P_{N+1}(t)$ and $\mathcal W_{N+1}(t)$, and thus inherits the  hybrid  structure.


	\item[(R3)] Initial data of the form $\mathcal P_N(0)y_0$ (resp., $\mathcal W_N(0)y_0$), under the action of the flow $\varPhi(\cdot)$, may lead to three nontrivial components. More precisely,
	given a fixed integer $N\geq 2$, there is  $j\in \mathbb N^+$ so that
	the following equality
	\begin{align*}
	\varPhi(t)(\mathcal P_{N}(0)e_{j})
	&=I_{N,1}(t)+I_{N,2}(t)+I_{N,3}(t)\\
	&:= \mathcal P_{N}(t) \big(
	\mathcal P_{N}(0)e_{j} \big)
	+ \mathcal W_{N}(t)  \big( \mathcal P_{N}(0)e_{j} \big)
	+ \mathfrak R_{N}(t) \big( \mathcal P_{N}(0)e_{j} \big),~~t\geq 0
	\end{align*}
	has the  property
	\begin{align}\label{march-three-nonzero}
	I_{N,1}(t),I_{N,2}(t),I_{N,3}(t) \neq 0
	~\text{in}~ L^2(\Omega)
	~\text{for a.e.}~t\geq 0.
	\end{align}
	(This is also true when $\mathcal P_N(0)e_{j}$ is replaced by $\mathcal W_N(0)e_{j}$ in the above statement.)

  \hskip 18pt Indeed, one can see  from \eqref{def-PN-HN-RN} and \eqref{RN-recursive-eq} that for each $j\geq 1$ and $t\geq 0$,
	\begin{align*}
	\left\{
	\begin{array}{lll}
	\mathcal P_N(t) \big( \mathcal P_N(0)e_{j} \big) &=& e^{-t\eta_{j} }
	\Big(1+\sum_{l=0}^{N-1} p_{l}(t)  \eta_{j}^{-l-1}  \Big) \Big(1+\sum_{l=0}^{N-1} p_{l}(0)  \eta_{j}^{-l-1}  \Big)  e_{j},\\
	\mathcal W_N(t) \big( \mathcal P_N(0)e_{j} \big)   &=& \Big(\sum_{l=0}^{N-1} h_{l}(t)
	\eta_{j}^{-l-1}  \Big)  \Big(1+\sum_{l=0}^{N-1} p_{l}(0)  \eta_{j}^{-l-1}  \Big) e_{j},\\
	\mathfrak R_N(t)\big( \mathcal P_N(0)e_{j} \big) &=& \Big(h_{N}(t) \eta_{j}^{-N-1}
	+  p_{N}(t)  e^{-t\eta_{j} }
	\eta_{j}^{-N-1} +  R_{N+1}(t,\eta_{j}) \eta_{j}^{-N-2}  \Big)\\
	& & ~~~~ \Big(1+\sum_{l=0}^{N-1} p_{l}(0)  \eta_{j}^{-l-1}  \Big) e_{j}.
	\end{array}
	\right.
	\end{align*}
	Since $\displaystyle\lim_{j\rightarrow+\infty} \eta_j=+\infty$, the above, along with the first inequality in \eqref{nontrivial-a.e.t-hl-pl} and the estimate \eqref{estimate-p-q-R}, implies
	\begin{align*}
		\mathcal P_N(\cdot)e_j,\mathcal W_N(\cdot)e_j,\mathfrak R_N(\cdot)e_j  \not\equiv 0
		~\text{for large}~j\in \mathbb N^+.
	\end{align*}
	At the same time,  the three functions on the left hand side above are analytic from $\mathbb R^+$ to $L^2(\Omega)$.  (For the first two functions, their analyticity follows from \eqref{def-PN-HN-RN} and \eqref{M=exp-hl-pl} and then the analyticity of  the last one is derived from Proposition \ref{cor-evolution-high-regularity}.) 	
	Therefore, they are non-trivial almost everywhere. This is exactly \eqref{march-three-nonzero}.  (In a similar way, \eqref{march-three-nonzero} can be verified in the case that $\mathcal P_N(0)e_j$ is replaced by $\mathcal W_N(0)e_j$.)
\end{itemize}

\end{remark}

\color{black}
\section{Conclusions and further comments}\label{sec-open-problems}

In this paper we have presented a decomposition for the flow
generated by the equation \eqref{our-system},
which reveals the hybrid parabolic-hyperbolic behavior of the flow. We have also described the  nature of each of the components in the decomposition;
  that has been illustrated through  an example.

A number of interesting issues could be considered in connection with the results and  methods developed
in this paper. Here, we briefly give some of them.

\begin{itemize}
  \item {\it  Smooth memory kernels.} It would be interesting  to analyze whether \eqref{2021-april-hl-pl-nontrivial},
  \eqref{yb-202104-wn-propagation} and
  \eqref{202104-varphi-wn-y0} hold under the assumption that $M\in C^\infty(\overline{ \mathbb R^+ })\setminus\{0\}$.

      \item{\it Decomposition with infinite series. } It would be interesting to obtain a meaningful decomposition without the intervention of the third remainder term.

  \item {\it Space-dependent memory kernels.} The extension of the results of this paper to the space-dependent memory kernels $M=M(t,x)$ is open.

  \item \textit{Memory kernels in the principal part of the model.} It would be interesting to extend our decomposition and analysis to the following two types of heat equations with memory kernels:
\begin{itemize}
	\item[($i$)]  $\partial_t y -  \Delta y - \int_0^t M(t-s) \Delta y(s)ds=0;$
	
	\item[($ii$)]  $\partial_t y -   \int_0^t M(t-s) \Delta y(s)ds=0,$
\end{itemize}
that are more relevant from an applied and modelling viewpoint.

  \item {\it Other equations with memory.} It would be interesting to extend this decomposition to other models such as   wave equations with memory kernels.

\end{itemize}

\section{Appendix}

\begin{proposition}\label{prop5.1}
For each $t\geq 0$ and $s\in \mathbb R$, it holds that
\begin{eqnarray}\label{5.1,2-17}
\varPhi(t)y_0=y(t;y_0),\;y_0\in \mathcal H^s
\end{eqnarray}
and that
 $\varPhi(t)$ belongs to  $\mathcal{L}(\mathcal H^s)$.
\end{proposition}

\begin{proof}
 Arbitrarily fix $s\in \mathbb R$.  Since $\{e^{tA}\}_{t\geq 0}$ is a $C_0$ semigroup over $\mathcal H^s$, we can use a standard method (see for instance \cite[Theorem 1.2 in Section 6.1, p. 184]{Pazy}) to  show that
 for each $y_0 \in \mathcal H^s$,
 the solution $y(\cdot;y_0)$ belongs to the space $C(\overline{ \mathbb R^+ }; \mathcal H^s)$.
 This, along with \eqref{varPhi-y-y0}, yields  \eqref{5.1,2-17}. Moreover, one can directly check that
  for each $t\geq 0$,
 $\Phi(t)\in\mathcal{L}(\mathcal H^s)$. This completes the proof.
\end{proof}

The proof of Lemma \ref{lem-convolution-estimates} is given as follows.

\begin{proof}[Proof of Lemma \ref{lem-convolution-estimates}]
By standard density arguments, it suffices to show
 that for each $k\in \mathbb{N}$, the following property $(\mathcal E_k)$ is true: For all kernels $\{M_l\}_{l=1}^j\subset C^\infty(\overline{\mathbb R^+})$,
\begin{align}\label{zhang-convolution-estimate}
	\Big| \frac{d^k}{dt^k} M_1 * \cdots * M_j (t) \Big|
	\leq  \mathcal C(j,k,t)
	\displaystyle\prod_{ l\leq q(j,k) } \| M_l \|_{C^{p(j,k)}([0,t])}
\displaystyle\prod_{ l> q(j,k) } \| M_l \|_{C^{p(j,k)-1}([0,t])},~~
		t> 0.
	\end{align}
Here and in what follows,
we set $\displaystyle\prod_{\emptyset}\cdot:=1$ and let
\begin{align}\label{july-0705-pq}
\left\{
\begin{array}{lll}
	p(j,k)&:=&\max\{k-j+1,1\},
\\
	q(j,k)&:=&\min\{k,j\},
\\
	\mathcal C(j,k,t) &:=& \displaystyle\sum_{l=j-1-k}^{j-1} \chi_{\mathbb N }(l) \frac{t^l}{l!},~~
t> 0.
\end{array}
\right.
\end{align}

Now we will use the induction to prove the above $(\mathcal E_{k})$
for  each $k\in \mathbb N$. To this end, we first check $(\mathcal E_{0})$.  Indeed,  for each  $j\in \mathbb N^+$ and each $\{M_l\}_{l=1}^j \subset C^{\infty}(\overline{\mathbb R^+})$,
we have that
\begin{align*}
	|M_1 * \cdots * M_j (t)|
\leq \bigg( \prod_{l=1}^j \|M_l\|_{C([0,t])} \bigg)
 \int_0^t\int_0^{t_1}\cdots\int_0^{t_{j-1}} dt_j\cdots dt_1
 = \frac{ t^{j-1} }{ (j-1)! }
	\displaystyle\prod_{l=1}^j \|M_l\|_{C([0,t])}.
\end{align*}
This, along with \eqref{july-0705-pq},  leads to \eqref{zhang-convolution-estimate} with $k=0$. Therefore $(\mathcal E_{0})$ is true.

Next, we will show
 $(\mathcal E_{k_0+1})$ for any $k_0\in \mathbb N$, under the assumption that $(\mathcal E_{k})$ holds for
 all $k\leq k_0$.
For this purpose, we
arbitrarily fix $k_0\in \mathbb N$, $j\in \mathbb N^+$ and  $\{M_l\}_{l=1}^j \subset C^{\infty}(\overline{\mathbb R^+})$.  	
Since \eqref{zhang-convolution-estimate}, with $j=1$,  holds clearly, we only need to focus on the situation that $j\geq 2$.
There are only two possibilities for $j$:  either $k_0\leq j-1$ or $k_0\geq j$.

In the case when  $k_0\leq j-1$, we have three observations: First, by direct computations, we find
\begin{align}\label{july-tj-0705-1}
	\frac{d^{k_0+1}}{dt^{k_0+1}} M_1 * \cdots * M_j  =& \frac{d^{k_0}}{dt^{k_0}}
	\Big( \frac{d}{dt} M_1 * \cdots * M_j    \Big)
	\nonumber\\
	=& \frac{d^{k_0}}{dt^{k_0}}
	\Big( M_1(0) M_2 * \cdots * M_j + M_1^\prime * M_2*\cdots *M_j \Big)
	\nonumber\\
	=& M_1(0) \frac{d^{k_0}}{dt^{k_0}} M_2 * \cdots * M_j +
	\frac{d^{k_0}}{dt^{k_0}} M_2 * \cdots * M_j * M_1^\prime .
\end{align}
Second, we  apply  $(\mathcal E_{k_0})$ twice to find that for each $t>0$,
\begin{align}\label{2.13,7.22}
	\Big| \frac{d^{k_0}}{dt^{k_0}} M_2 * \cdots * M_j(t) \Big|
	\leq  \mathcal C(j-1,k_0,t)	\bigg( \displaystyle\prod_{2\leq l\leq k_0+1} \| M_l \|_{C^{1}([0,t])} \bigg)
	\bigg( \displaystyle\prod_{l>k_0+1} \| M_l \|_{C([0,t])} \bigg);
\end{align}
\begin{align}\label{2.14,7.22}
	\Big| \frac{d^{k_0}}{dt^{k_0}} M_2 * \cdots * M_j*M_1^\prime(t) \Big|
	\leq  \mathcal C(j,k_0,t) 	\bigg( \displaystyle\prod_{2\leq l \leq k_0+1} \| M_l \|_{C^{1}([0,t])} \bigg)
	\bigg( \displaystyle\prod_{l> k_0+1} \| M_l \|_{C([0,t])} \bigg)
	\|M_1^\prime\|_{C([0,t])}.
\end{align}
Third, from  the third definition in (\ref{july-0705-pq}), we see
\begin{align}\label{july-tj-0705-3}
	\max\Big\{ \mathcal C(j-1,k_0,t), ~\mathcal C(j,k_0,t) \Big\}
	\leq \mathcal C(j,k_0+1,t),~~
t>0 .
\end{align}
Now, from  (\ref{july-tj-0705-1}), (\ref{2.13,7.22}), (\ref{2.14,7.22})
and (\ref{july-tj-0705-3}), it follows that  for each $t>0$,
\begin{align*}
	\Big| \frac{d^{k_0+1}}{dt^{k_0+1}} M_1 * \cdots * M_j (t) \Big|
	\leq&  \mathcal C(j,k_0+1,t) \Big( \|M_1\|_{C([0,t])} + \|M_1^\prime\|_{C([0,t])} \Big)
	\nonumber\\
	& \quad \times
	\bigg( \displaystyle\prod_{2\leq l \leq k_0+1} \| M_l \|_{C^{1}([0,t])} \bigg)
	\bigg( \displaystyle\prod_{l>k_0+1} \| M_l \|_{C([0,t])} \bigg).
\end{align*}
This, along with \eqref{july-0705-pq},  leads to \eqref{zhang-convolution-estimate} (with $k=k_0+1\leq j$). Therefore, 	$(\mathcal E_{k_0+1})$ is true when $k_0\leq j-1$.

In the case that  $k_0\geq j$, we see  from (\ref{july-0705-pq}) that
\begin{align*}
	\begin{cases}
		p(j,k_0+1)=p(j-1,k_0)=p(j,k_0)+1=k_0-j+2;
		\\
		q(j,k_0+1)=q(j,k_0)=q(j-1,k_0)+1=j.
	\end{cases}
\end{align*}
Then by the similar arguments as those in (\ref{july-tj-0705-1})-(\ref{july-tj-0705-3}), one can get that for each $t>0$,
\begin{align*}
	\Big| \frac{d^{k_0+1}}{dt^{k_0+1}} M_1 * \cdots * M_j (t) \Big|
	\leq& \mathcal C(j,k_0+1,t) \Big( \|M_1\|_{C([0,t])} + \|M_1^\prime\|_{C^{p(j,k_0)}([0,t])} \Big)
\nonumber\\
&~~~\times
	\displaystyle\prod_{l=2}^{j} \| M_l \|_{C^{p(j-1,k_0)}([0,t])}
	\nonumber\\
	=& \mathcal C(j,k_0+1,t) \displaystyle\prod_{l=1}^{j} \| M_l \|_{C^{p(j,k_0+1)}([0,t])}.
\end{align*}
This, together with \eqref{july-0705-pq},  leads to \eqref{zhang-convolution-estimate} (with $k=k_0+1> j$). Therefore, 	$(\mathcal E_{k_0+1})$ holds  in the case that $k_0\geq j$.

Hence, we complete the proof of Lemma \ref{lem-convolution-estimates}.	
\end{proof}

\bigskip
\footnotesize
\noindent\textit{Acknowledgments.}
The  authors would like to gratefully thank Dr. Huaiqiang Yu
and Dr. Christophe Zhang for their valuable comments and suggestions.

The first and second authors were supported by the National Natural Science Foundation of China under grants 11971022, 11801408 and 12171359.

The  third author has been funded by the Alexander von Humboldt-Professorship program, the European Research Council (ERC) under the European Union's Horizon 2020 research and innovation programme (grant agreement No. 694126-DyCon), the Transregio 154 Project ``Mathematical Modeling, Simulation and Optimization Using the Example of Gas Networks” of the German DFG, the grant MTM2017-92996 of MINECO (Spain), ELKARTEK project KK-2018/00083 ROAD2DC of the Basque Government and the Marie Sklodowska-Curie grant agreement No. 765579-ConFlex.


\begin{thebibliography}{SK}




\normalsize
\baselineskip=17pt






\bibitem{Amendola-Fabriio-Golden}  G. Amendola, M. Fabrizio and J. Golden. Thermodynamics of materials with memory: Theory and applications. Springer, New York, 2012.

\bibitem{Boltzmann-1} L. Boltzmann. Zur theorie der elastischen nachwirkung. Wien. Ber. 70 (1874) 275--306.

\bibitem{Boltzmann-2} L. Boltzmann. Zur theorie der elastischen nachwirkung. Wien. Ber. 5 (1878) 430--432.

\bibitem{Cattaneo} C. Cattaneo. A form of heat conduction equation which eliminates the paradox of instantaneous propagation. Compute. Rendus. 247 (1958) 431-433.


\bibitem{Chaves-Silva-Zhang-Zuazua} F. Chaves-Silva, X. Zhang and E. Zuazua. Controllability of evolution equations with memory. SIAM J. Control Optim. 55 (2017) 2437-2459.

\bibitem{Christensen} R. Christensen. Theory of viscoelasticity, an introduction. Academic Press, New York, 1982.

\bibitem{Coleman-Gurtin} B.  Coleman and M.  Gurtin. Equipresence and constitutive equations for rigid heat conductors. Z. Angew. Math. Phys. 18 (1967) 199--208.

\bibitem{Dafermos} C. Dafermos.  Asymptotic Stability in Viscoelasticity. Arch.  Ration. Mech. Anal. 37 (1970) 297--308.



\bibitem{Farbizio-Gogi-Pata} M. Fabrizio, C. Giorgi and V. Pata. A new approach to equations with memory. Arch. Ration. Mech. Anal. 198 (2010) 189-232.

%
%

\bibitem{Fu-Yong-Zhang} X. Fu, J. Yong and X. Zhang. Controllability and observability of the heat equations with hyperbolic memory kernel. J. Differential Equations 247 (2009) 2395--2439.



\bibitem{Gurtin-Pipkin} M. Gurtin and A. Pipkin. A general theory of heat conduction with finite wave speeds. Arch. Ration. Mech. Anal. 31 (1968) 113-126.




%
\bibitem{Hormander-3} L. H\"{o}rmander. The Analysis of Linear Partial Differential Operators, Vol. 3, Springer-Verlag, 2007.



%


\bibitem{Lu-Zhang-Zuazua} Q. L\"{u}, X. Zhang and E. Zuazua. Null controllability for wave equations with memory. J. Math. Pures Appl. 108 (2017) 500-531.


\bibitem{Maxwell}  J.  Maxwell. On the dynamical theory of gases. Phil. Trans. Roy.
Soc. London 157 (1867) 49--88.

\bibitem{SIMON1}   M. Reed and B. Simon. Methods of Modern Mathematical Physics I: Functional Analysis, Academic Press, 1980.

\bibitem{SIMON2}   M. Reed and B. Simon. Methods of Modern Mathematical Physics IV:  Analysis of Operators, Academic Press, 1978.

%




\bibitem{Pandolfi} L. Pandolfi. Linear systems with persistent memory: An overview of the biblography on controllability. arXiv: 1804.01865v1 (2018), preprint.

\bibitem{Pazy} A. Pazy. Semigroups of Linear Operators and Application to Partial Differential Equations. Springer-Verlag, New York, Inc., 1983.


%
%
%
%
\bibitem{Volterra-1} V. Volterra. Sur les \'{e}quations int\'{e}gro-diff\'{e}rentielles et leurs applications. Acta Math.
35 (1912) 295--356.

\bibitem{Volterra-2} V. Volterra.  Le\c{c}ons sur les fonctions de lignes. Gauthier-Villars, Paris, 1913.





\bibitem{WZZ-2} G. Wang, Y. Zhang and E. Zuazua. Reachable subspaces, control regions and heat equations with memory. arXiv:2101.10615v1, preprint.




\end{thebibliography}
\end{document}